\documentclass[12pt]{amsart}
\usepackage{latexsym}
\usepackage{amssymb, amsmath,mathrsfs}
\usepackage[left=2.5cm,right=2.5cm,top=2.5cm, bottom=2.5cm ]{geometry}
\usepackage[OT2,T1]{fontenc}

\usepackage{setspace}
\usepackage{graphicx}

\usepackage[pagebackref,hypertexnames=false, colorlinks, citecolor=red, linkcolor=red]{hyperref}

\newtheorem{theorem}{Theorem}[section]
\newtheorem{lemma}[theorem]{Lemma}

\newtheorem{corollary}[theorem]{Corollary}
\newtheorem{proposition}[theorem]{Proposition}

\theoremstyle{remark}
\newtheorem{remark}[theorem]{Remark}
\newtheorem{example}[theorem]{Example}

\DeclareSymbolFont{cyrletters}{OT2}{wncyr}{m}{n}
\DeclareMathSymbol{\Sha}{\mathalpha}{cyrletters}{"58}
\author[Bickel]{Kelly Bickel$^\dagger$}
\address{Department of Mathematics, Bucknell University, 380 Olin Science Building, Lewisburg, PA 17837, USA.}
\email{kelly.bickel@bucknell.edu}
\thanks{$\dagger$ Research supported in part by National Science Foundation
DMS grant \#1448846.}

\author[Gorkin]{Pamela Gorkin$^\ddagger$}
\address{Department of Mathematics, Bucknell University, 380 Olin Science Building, Lewisburg, PA 17837, USA.}
\email{pgorkin@bucknell.edu}
\thanks{$\ddagger$ Research supported in part by Simons Foundation Grant 243653.}

\begin{document}
\raggedbottom
\title{Compressions of the shift on the bidisk and their numerical ranges}
\date{\today}


\maketitle
\begin{abstract}
We consider two-variable model spaces associated to rational inner functions  on the bidisk, which always possess canonical $z_2$-invariant subspaces $\mathcal{S}_2.$ A particularly interesting compression of the shift is the compression of multiplication by $z_1$ to $\mathcal{S}_2$, namely $ S^1_{\Theta}:= P_{\mathcal{S}_2} M_{z_1} |_{\mathcal{S}_2}$. We show that these compressed shifts are unitarily equivalent to matrix-valued Toeplitz operators with well-behaved symbols and characterize their numerical ranges and radii. We later specialize to particularly simple rational inner functions and study the geometry of the associated numerical ranges, find formulas for the boundaries, answer the zero inclusion question, and determine whether the numerical ranges are ever circular.

\end{abstract}
\section{Introduction}

\subsection{One-variable setting.} Let $\theta$ be an inner function on the Hardy space $H^2(\mathbb{D})$ and let $\mathcal{K}_{\theta} := H^2(\mathbb{D}) \ominus \theta H^2(\mathbb{D})$ be its model space. The associated compressions of the shift $S_\theta := P_{\theta} M_z|_{\mathcal{K}_{\theta}}$ (multiplication by $z$ followed by the orthogonal projection onto $\mathcal{K}_{\theta}$) have played a pivotal role in both operator and function theory. Indeed, allowing $\theta$ to be operator valued, the famous Sz.-Nagy--Foias model theory says: every completely nonunitary, $C_0$ contraction is unitarily equivalent to a compression of the shift $S_\theta$ on a model space $\mathcal{K}_{\theta}$ \cite{Snf10}. 

If the inner function is a finite Blaschke product $B$, i.e.
\[ B(z) = \prod_{i=1}^m \frac{ z-\alpha_i}{1-\bar{\alpha}_iz}, \qquad \text{ where } \alpha_1, \dots, \alpha_m \in \mathbb{D},\]
then the associated compression of the shift $S_B$ is quite well behaved. Indeed, the matrix of $S_{B}$ with respect to a basis called the Takenaka-Malmquist-Walsh basis $\{f_1, \dots, f_m\}$, see \cite{gr15}, is the upper triangular matrix $M_B$ given entry-wise by
  \begin{equation} \label{eqn:onevar}
  (M_{B})_{ij}:= \left \langle S_B f_j, f_i\right \rangle_{\mathcal{K}_{\theta}}  = \left \{ \begin{array}{cc}
  \alpha_i & \text{ if } i=j; \\
   \prod_{k=i+1}^{j-1} (-\overline{\alpha_k})(1-|\alpha_i|^2)^{1/2} (1-|\alpha_j|^2)^{1/2}  &\text{ if } i <j; \\
 0 & \text{ if } i >j.
  \end{array}
  \right.
  \end{equation}
For this matrix, see the survey \cite[pp.~180]{gw03}. Formula \eqref{eqn:onevar} allows one to answer many natural questions about the structure of $S_{\theta}.$ Answers concerning the numerical range and radius are particularly nice. Namely, if $T: \mathcal{H} \rightarrow \mathcal{H}$ is a bounded operator on a Hilbert space $\mathcal{H}$, then the \emph{numerical range} of $T$ is the set 
\[ \mathcal{W}(T) := \left \{ \langle T h, h \rangle_{\mathcal{H}} : \| h\|_{\mathcal{H}} = 1 \right \}\]
and the \emph{numerical radius} of $T$ is the number 
\[ w(T) := \sup \left\{ |\lambda | : \lambda \in \mathcal{W}(T) \right\}.\]
Discussion of these sets for compressed shifts associated to finite Blaschke products requires some geometry. Recall that Poncelet's closure theorem says:~given two ellipses with one contained in the other, if there is an $N$-sided polygon circumscribing the smaller ellipse that has all of its vertices on the larger ellipse, then for every $\lambda$ on the larger ellipse there is such an $N$-sided circumscribing polygon with a vertex at $\lambda$, see \cite[Section $5$]{gw03}.  Similarly, for $N \ge 3$, we say a curve $\Gamma \subset \mathbb{D}$ satisfies the \emph{$N$-Poncelet property} if for 
 each point $\lambda \in \partial \mathbb{D}=\mathbb{T},$ there is an $N$-sided polygon circumscribing $\Gamma$ with one vertex at $\lambda$ and all other vertices on $\mathbb{T}$, see \cite[pp. 182]{gw03}.
 
 Surprisingly, Poncelet curves have close ties to numerical ranges. Indeed, let $B$ be a finite Blaschke product of degree $m>1$. Then, as shown by Mirman \cite{m98} and Gau and Wu \cite{gw98},  the boundary $\partial \mathcal{W}(S_B)$ actually possesses the $(m+1)$-Poncelet property. The idea behind the proof is quite intuitive; the inscribing polygons are in one-to-one correspondence with the unitary $1$-dilations of $S_B$, which are obtained from \eqref{eqn:onevar}. Moreover the vertices of the polygons are exactly the eigenvalues of the  unitary $1$-dilations, and because $\partial \mathcal{W}(S_B)$ is strictly contained in $\mathbb{D}$, the numerical radius $w(S_B)$  is always strictly less than $1$. For a detailed exploration of Poncelet ellipses for $B$ a degree-$3$ Blaschke product, see \cite{DGM}, and for similar results concerning infinite Blaschke products, see \cite{cgp09}. 
 
 In what follows, we study these and other geometric properties of numerical ranges and radii of compressions of shifts on the bidisk $\mathbb{D}^2$. 

\subsection{Two-variable setting} For the two-variable case, let $\Theta$ be an inner function on $\mathbb{D}^2$, namely a function holomorphic on $\mathbb{D}^2$ whose boundary values satisfy $|\Theta(\tau)|=1$ for almost every $\tau \in \mathbb{T}^2$. Then let $\mathcal{K}_{\Theta}$ be the associated two-variable model space defined by
\[ 
\mathcal{K}_{\Theta} := H^2(\mathbb{D}^2) \ominus \Theta H^2(\mathbb{D}^2)= \mathcal{H}\left( \frac{1 - \Theta(z) \overline{\Theta(w)}}{(1-z_1\bar{w}_1)(1-z_2\bar{w}_2)} \right),\]
where $\mathcal{H}(K)$ denotes the reproducing kernel Hilbert space with reproducing kernel $K$. In this paper, we use $\Theta$ to denote two-variable inner functions and $\theta$ for simpler, often one-variable inner functions. In this setting, one natural compression of the shift is the operator
\[  \widetilde{S}^1_{\Theta} := P_{\Theta} M_{z_1}|_{ \mathcal{K}_{\Theta}},\] 
where $P_{\Theta}$ denotes the orthogonal projection of $H^2(\mathbb{D}^2)$ onto $\mathcal{K}_{\Theta}$ and $M_{z_1}$ is multiplication by $z_1$.  Although we explicitly study $ \widetilde{S}^1_{\Theta}$, symmetric results will hold for a similarly-defined $\widetilde{S}_{\Theta}^2.$ 

As in the one-variable discussion, we restrict attention to $\Theta$ that are both rational and inner. Section \ref{sec:rational} includes most needed details about rational inner functions, but discussing our main results will require some notation. First, the degree of $\Theta$, denoted $\deg \Theta = (m,n),$ is defined as follows:~write $\Theta=\frac{q}{p}$ with $p$ and $q$ polynomials with no common factors. Then $m$ is the highest degree of $z_1$ and $n$ the highest degree of $z_2$ appearing in either $p$ or $q$.  Moreover, if  $\Theta$ is rational inner with $\deg \Theta = (m,n)$, then there is an (almost) unique polynomial $p$ with no zeros on $\mathbb{D}^2$ such that 
$\Theta = \frac{\tilde{p}}{p},$
where $\tilde{p}(z) := z_1^m z_2^n \overline{p( \frac{1}{\bar{z}_1}, \frac{1}{\bar{z}_2})}$ and $p$ and $\tilde{p}$ share no common factors. See \cite{ams06, Rud69} for details.

Our goal is to study the numerical range of a general compression of the shift $ \widetilde{S}^1_{\Theta}$ associated to a rational inner function $\Theta$. Unfortunately, the question
\begin{center} ``What are the properties of $\mathcal{W}( \widetilde{S}^1_{\Theta})$?'' \end{center}
often has a trivial answer. To observe the problem, one can decompose $\mathcal{K}_{\Theta}$ as  
\begin{equation} \label{eqn:sub} \mathcal{K}_{\Theta} = \mathcal{S}_1 \oplus \mathcal{S}_2,\end{equation}
where $\mathcal{S}_1$ and $\mathcal{S}_2$ are respectively $M_{z_1}$- and $M_{z_2}$-invariant. There are canonical ways to obtain such decompositions, and details are provided in 
Section \ref{sec:rational}. If $\mathcal{S}_1$ is nontrivial, then $ \widetilde{S}^1_{\Theta}|_{\mathcal{S}_1} = M_{z_1}|_{\mathcal{S}_1}$ and one can further show that
\[ \text{Clos} \left \{ \left \langle  \widetilde{S}^1_{\Theta} f, f \right \rangle_{\mathcal{K}_{\Theta}} : \| f\|_{\mathcal{K}_{\Theta}} = 1, f \in \mathcal{S}_1 \right \} = \overline{\mathbb{D}}.\]
Then since $ \widetilde{S}^1_{\Theta}$ is a contraction, we can conclude that $\text{Clos}( \mathcal{W}( \widetilde{S}^1_{\Theta}))$ equals $\overline{\mathbb{D}};$ see Lemma \ref{lem:range} for details. Because of this, we compress $ \widetilde{S}^1_{\Theta}$ to the $M_{z_2}$-invariant subspace $\mathcal{S}_2$ from \eqref{eqn:sub} and study this more interesting compression of the shift:
\begin{equation} \label{eqn:compression}  S^1_{\Theta} : = P_{\mathcal{S}_2}  \widetilde{S}^1_{\Theta} |_{\mathcal{S}_2} = P_{\mathcal{S}_2} M_{z_1} |_{\mathcal{S}_2}.\end{equation}

\subsection{Outline and Main Results}
This paper studies the structure of the compression of the shift  $ S^1_{\Theta}$ defined in \eqref{eqn:compression} and the geometry of its numerical range. It is outlined as follows: in Section \ref{sec:rational}, we detail needed results about rational inner functions and their model spaces on the bidisk. In Section \ref{sec:structure}, we obtain most of our structural results about  $ S^1_{\Theta}$ and its numerical range, while in Section \ref{sec:examples}, we illustrate the results from Section \ref{sec:structure} with examples. In Sections \ref{Zero} and \ref{sec:boundary}, we study the geometry of  the numerical ranges $\mathcal{W}( S^1_{\Theta})$ associated to simple rational inner functions; Section \ref{Zero} addresses the zero inclusion question, and Section \ref{sec:boundary} examines the shape of the boundary of the numerical range.

Before stating our main results, we require the following notation: $H^2_2(\mathbb{D})$ denotes the one-variable Hardy space with independent variable $z_2$ and $H^2_2(\mathbb{D})^m : = \bigoplus_{i=1}^m H^2_2(\mathbb{D})$ denotes the space of vector-valued functions $\vec{f} = (f_1, \dots, f_m)$ with each $f_i \in H^2_2(\mathbb{D})$. Define $L_2^2(\mathbb{T})^m$ analogously, and let $F$ be a bounded $m\times m$ matrix-valued function defined for almost every $ z_2 \in \mathbb{T}$. Then the \emph{$z_2$-matrix-valued Toeplitz operator with symbol F} is the operator
\begin{equation} \label{eqn:Toeplitz}
T_F : H^2_2(\mathbb{D})^m \rightarrow H^2_2(\mathbb{D})^m  \ \ \text{ defined by }  \ \  T_F \vec{f}  = P_{H_2^2(\mathbb{D})^m} \big( F  \vec{f}  \ \big),
\end{equation}
where $P_{H_2^2(\mathbb{D})^m} $ is the orthogonal projection of $L_2^2(\mathbb{T})^m$ onto $H_2^2(\mathbb{D})^m$.

Then, in Section \ref{sec:structure}, we show that each $ S^1_{\Theta}$ is unitarily equivalent to a $z_2$-matrix-valued Toeplitz operator with a well-behaved symbol as follows:

\medskip
\noindent
{\bf Theorem~\ref{thm:unitary}.} {\it Let $\Theta = \frac{\tilde{p}}{p}$ be rational inner of degree $(m,n)$ and let $\mathcal{S}_2$  be as in \eqref{eqn:sub}. Then there exists an $m\times m$  matrix-valued function $M_{\Theta}$, with entries that are rational functions of $\bar{z}_2$ and continuous on $\overline{\mathbb{D}},$ such that
\[ S^1_{\Theta} = \mathcal{U} \ T_{M_{\Theta}} \ \mathcal{U}^*, \]
where $\mathcal{U}: H^2_2(\mathbb{D})^m \rightarrow \mathcal{S}_{2}$ is a unitary operator defined in \eqref{eqn:unitary}.}
\medskip

One can view Theorem \ref{thm:unitary}  as a generalization of the formula \eqref{eqn:onevar} for the matrix of a compressed shift associated to a Blaschke product. As in the one-variable setting,  this structural result gives information about the numerical range of $ S^1_{\Theta},$ namely:

\medskip
\noindent
{\bf Corollary~\ref{thm:nr}.} {\it Let $\Theta = \frac{\tilde{p}}{p}$ be rational inner of degree $(m,n)$, let $\mathcal{S}_2$  be as in \eqref{eqn:sub}, and let $M_{\Theta}$ be as in Theorem \ref{thm:unitary}. Then
\[ \text{Clos}\left(\mathcal{W}\left(S^1_{\Theta}  \right)\right) =  \text{Conv} \Big(  \bigcup_{\tau \in \mathbb{T}} \mathcal{W}\left(M_{\Theta}(\tau) \right)  \Big).\]}

Here ``Clos'' denotes the closure and ``Conv'' denotes the convex hull of the given sets.  Then Corollary \ref{thm:nr} says that $\mathcal{W}( S^1_{\Theta} )$ is built out of numerical ranges of specific $m \times m$ matrices. We also connect $\mathcal{W}( S^1_{\Theta} )$ to the numerical ranges of compressed shifts associated to degree-$m$ Blaschke products, see Theorem \ref{thm:generalnr}. This result is particularly important because it links the rich one-variable theory to this two-variable setting. For example, it implies that  Clos($\mathcal{W}( S^1_{\Theta} )$) is the closed convex hull of a union of sets whose boundaries satisfy the $(m+1)$-Poncelet property.  Amongst other results, we also combine Theorem \ref{thm:generalnr} with one-variable facts to characterize when the numerical radius is maximal:

\medskip
\noindent
{\bf Theorem~\ref{thm:numericalradius}.} {\it Let $\Theta = \frac{\tilde{p}}{p}$ be rational inner of degree $(m,n)$ and let $\mathcal{S}_2$  be as in \eqref{eqn:sub}. Then the numerical radius $w\big( S^1_{\Theta}  \big) =1$ if and only if $\Theta$ has a singularity on $\mathbb{T}^2$.}
\medskip

This theorem shows that certain one-variable properties do not (in general) hold in this two-variable setting. Indeed, as $N$-Poncelet sets cannot touch $\mathbb{T}$, this implies that  if $\Theta$ has a singularity on $\mathbb{T}^2$, then the boundary of $\mathcal{W}( S^1_{\Theta} )$ does not satisfy an $N$-Poncelet property.

In Section \ref{sec:examples}, we illustrate these theorems with examples. We consider  $\Theta := \prod_{i=1}^m \theta_i,$ where each $\theta_i$ is a degree  $(1,1)$ rational inner function  with a singularity on $\mathbb{T}^2.$ Specifically, we decompose the associated $\mathcal{K}_{\Theta}$ into concrete $M_{z_1}$- and $M_{z_2}$-invariant subspaces $\mathcal{S}_1$ and $\mathcal{S}_2$, find an orthonormal basis of $\mathcal{H}(K_2) := \mathcal{S}_2 \ominus M_{z_2} \mathcal{S}_2$, and use that to compute explicitly the matrix-valued function $M_{\Theta}$ from Theorem \ref{thm:unitary}. Proposition \ref{prop:productform} contains the decomposition of $\mathcal{K}_{\Theta}$ and the orthonormal basis of $\mathcal{H}(K_2)$, while Theorem \ref{thm:nr1} contains the formula for $M_{\Theta}.$

In Section~\ref{Zero}, we restrict attention to  $\Theta = \theta_1 \theta_2,$ where each $\theta_i$ is a degree  $(1,1)$ rational inner function  with a singularity on $\mathbb{T}^2.$  For these $\Theta,$ Theorem \ref{thm:nr1} gives a formula  for $M_{\Theta}$, which shows that $\mathcal{W}( S^1_{\Theta})$ is basically the convex hull of an infinite union of ellipses with specific foci and axes. This information allows us to study the geometry of these numerical ranges and in particular, investigate the classical problem:

\begin{center} ``When is zero in the numerical range $\mathcal{W}( S^1_{\Theta})$?''\end{center} 

An answer to the zero inclusion question often yields useful information. For example, the numerical range of a compact operator $T$ is closed if and only if $0 \in \mathcal{W}(T)$, \cite{BGS}. Bourdon and Shapiro \cite{BS} studied the zero inclusion question for composition operators showing, among other things, that the numerical range of a composition operator other than the identity always contains zero in the closure of the numerical range. More recently, Higdon \cite{HIGDON} showed that if $\varphi$ is a holomorphic self-map of $\mathbb{D}$ with Denjoy-Wolff point on the unit circle that is not a linear fractional transformation, then zero is an interior point of the numerical range of the composition operator $\mathcal{C}_\varphi$. 

In our setting, we obtain several results related to the zero inclusion question for $\mathcal{W}( S^1_{\Theta} )$. First, in Proposition~\ref{lem:general}, we obtain two conditions guaranteeing that zero is in this numerical range; these conditions involve the foci of the elliptical disks comprising $\mathcal{W}( S^1_{\Theta} )$.  We then impose additional restrictions on the coefficients of the rational inner function. Under these restrictions, in Proposition ~\ref{lem:positive_eigenvalues}, we obtain necessarily and sufficient conditions for both zero to be in the interior and zero to be in the boundary of the numerical range.

In Section \ref{sec:boundary}, we further study the shape of the numerical range $\mathcal{W}( S^1_{\Theta} )$. Due to the complexity of the computations, we only consider rational inner functions of the form  $\Theta = \theta^2_1$, where $\theta_1 = \frac{\tilde{p}}{p}$ for a polynomial $p(z)=a -z_1 + c z_2$ with no zeros on $\mathbb{D}^2$, a zero on $\mathbb{T}^2$, and $a, c >0$.   We initially consider the question: 
\begin{center} ``When is the numerical range $\mathcal{W}( S^1_{\Theta})$ circular?''\end{center}

For more general operators, this question has a long and interesting history. For example, Anderson showed that if an $m \times m$ matrix $M$ has the property that $\mathcal{W}(A)$ is contained in $\overline{\mathbb{D}}$ and there are more than $m$ points with modulus $1$ in the numerical range, then $\mathcal{W}(A) = \mathbb{D}$ and zero is an eigenvalue of $A$ of multiplicity at least $2$. In \cite{Wu}, Wu extends these results. 

We show that for our restricted class of rational inner functions, which seem to be the ones most likely to produce a circular numerical range, $\mathcal{W}( S^1_{\Theta})$ is never circular.  We then interpret the union of circles comprising  $\mathcal{W}( S^1_{\Theta})$ as a family of curves. Using the theory of envelopes, we are able to obtain a precise description of the boundary of the numerical range. The exact parameterization is given in Theorem \ref{thm:boundary}.  We refer the reader to \cite{Wu} for more information and other references about this question.

\subsection*{Acknowledgements} The authors gratefully acknowledge Institut Mittag-Leffler, where this work was initiated.
 The authors would also like to thank Elias Wegert for sharing a simple method for computing the envelope of a family of curves. 
\section{Rational Inner Functions \& Model Spaces} \label{sec:rational}

Let $\Theta$ be a rational inner function on $\mathbb{D}^2$ with $\deg \Theta = (m,n)$. As mentioned earlier, there is a basically unique polynomial $p$ with no zeros on $\mathbb{D}^2$ such that $\Theta = \frac{\tilde{p}}{p}$, where $\tilde{p}(z) = z_1^m z_2^n \overline{p( \frac{1}{\bar{z}_1}, \frac{1}{\bar{z}_2})}$ and $\tilde{p}, p$ have no common factors.

An application of B\'ezout's Theorem implies that $p, \tilde{p}$ have at most $2mn$ common zeros, including intersection multiplicity and moreover, they will have exactly $2mn$ common zeros if $\deg p = \deg \tilde{p}.$  Moreover, one can easily check that $p$ and $\tilde{p}$ have the same zeros on $\mathbb{T}^2$. Then as common zeros of $p$ and $\tilde{p}$ on $\mathbb{T}^2$ have even intersection multiplicity, $p$ can vanish at no more than $mn$ points on $\mathbb{T}^2$. For further details and proofs of these comments, see \cite{k14}. Then, an application of Theorem $4.9.1$ in \cite{Rud69} implies that $p$ also has no zeros on $(\mathbb{D} \times \mathbb{T}) \cup (\mathbb{T} \times \mathbb{D})$. 

If $\Theta$ is an inner function (not necessarily rational), the structure of the model space $\mathcal{K}_{\Theta}$ is also quite interesting. As mentioned earlier, there are canonical ways to decompose every nontrivial $\mathcal{K}_{\Theta}$ into subspaces that are $M_{z_1}$- and $M_{z_2}$-invariant, or equivalently, $z_1$- and $z_2$-invariant, as in \eqref{eqn:sub}.  For example, as discussed in  \cite{bsv05,bk13}, if you set $\mathcal{S}_1^{max}$ to be the maximal subspace of $\mathcal{K}_{\Theta}$ invariant under $M_{z_1}$,  then $\mathcal{S}^{max}_1$ is clearly $z_1$-invariant and $\mathcal{S}_2^{min} := \mathcal{K}_{\Theta} \ominus \mathcal{S}^{max}_1$ is $z_2$-invariant. One can similarly define $\mathcal{S}_2^{max}$ and $\mathcal{S}_1^{min}.$

Given any such subspaces $\mathcal{S}_1$ and $\mathcal{S}_2$ with $\mathcal{K}_{\Theta} = \mathcal{S}_1 \oplus \mathcal{S}_2$ and each $\mathcal{S}_j$ $z_j$-invariant, it makes sense to define reproducing kernels  $K_1$, $K_2: \mathbb{D}^2 \times \mathbb{D}^2 \rightarrow \mathbb{C}$ by
\begin{equation} \label{eqn:kernel}
\mathcal{H}(K_1) = \mathcal{S}_1 \ominus z_1 \mathcal{S}_1  \ \ \text{ and } \ \
\mathcal{H}(K_2) = \mathcal{S}_2 \ominus z_2 \mathcal{S}_2.
\end{equation}
The resulting pair of kernels $(K_1, K_2)$ is called a pair of \emph{Agler kernels} of $\Theta$ because the kernels satisfy the equation
\begin{equation} \label{eqn:ad} 1-\Theta(z) \overline{\Theta(w)} = (1-z_1 \bar{w}_1)K_2(z,w) + (1-z_2 \overline{w}_2) K_1(z,w),\end{equation}
for all $z,w \in \mathbb{D}^2.$ Indeed, any positive semidefinite kernels $(K_1, K_2)$ satisfying \eqref{eqn:ad} are called \emph{Agler kernels of $\Theta$} and the equation \eqref{eqn:ad} is called an  
\emph{Agler decomposition of $\Theta.$} The existence of Agler decompositions was first proved by Agler in \cite{ag90}.
 
If $\Theta$ is rational inner, there are close connections between the properties of $\Theta$ and the structure of the Hilbert spaces $\mathcal{H}(K_1)$ and $\mathcal{H}(K_2).$ The following result appears in \cite{k11}  and follows by an examination of the degrees and singularities of the functions in \eqref{eqn:ad}:

\begin{theorem} \label{thm:dim} Let $\Theta = \frac{ \tilde{p}}{p}$ be a rational inner function of degree $(m,n)$ and let $K_1, K_2$ be Agler kernels of $\Theta$ as in \eqref{eqn:ad}. Then $\dim \mathcal{H}(K_1)$, $\dim \mathcal{H}(K_2)$ are both finite. Moreover, if $g$ is a function in $\mathcal{H}(K_1),$ then $g = \frac{r}{p}$ where $\deg r \le (m,n-1)$ and  if $f$ is a function in $\mathcal{H}(K_2)$, then $f = \frac{q}{p}$ where $\deg q \le (m-1,n).$
\end{theorem}

Define the following exceptional set
\begin{equation} \label{eqn:exceptional} E_{\Theta} := \Big\{ \tau \in \mathbb{T}: \exists \ \tau_1 \in \mathbb{T} \text{ such that } p(\tau_1, \tau)=0 \Big\}.  \end{equation}
By the above comments about $\Theta$, the set $E_{\Theta}$ is necessarily finite. For $\tau \in \mathbb{T}$, define the slice function $\Theta_{\tau}$ by $\Theta_{\tau} \equiv \Theta(\cdot, \tau).$ Then $\Theta_{\tau}$ is a finite Blaschke product and in what follows, $\mathcal{K}_{\Theta_{\tau}}$ will denote the one-variable model space associated to $\Theta_{\tau}.$

The following result is proved for Hilbert spaces arising from canonical decompositions of $\mathcal{K}_{\Theta}$ in \cite{bk13, w10}.  Specifically,  see Theorems 1.6-1.8 in \cite{bk13} as well as Proposition $2.5$ in \cite{w10}. Here, we include the proof for more general decompositions of $\mathcal{K}_{\Theta}$, which basically mirrors the ideas appearing in \cite{bk13}. 

\begin{theorem} \label{thm:rational} Let $\Theta = \frac{ \tilde{p}}{p}$ be a rational inner function of degree $(m,n)$ and let $K_1, K_2$ be defined as in \eqref{eqn:kernel}. Then for any $\tau \in \mathbb{T} \setminus E_{\Theta}$,  $\Theta_\tau$ is a Blaschke product with $\deg \Theta_ {\tau}=m$ and the restriction map $\mathcal{J}_{\tau}: \mathcal{H}(K_2) \rightarrow \mathcal{K}_{\Theta_{\tau}}$ defined by $\mathcal{J}_{\tau} f = f(\cdot, \tau)$  is unitary. Furthermore,  $\dim \mathcal{H}(K_2) =m.$  The analogous statements hold for $\mathcal{H}(K_1).$
 \end{theorem}

\begin{proof} By Theorem \ref{thm:dim}, $\dim \mathcal{H}(K_2)=M$ for some $M \in \mathbb{N}.$ We will later conclude that $M=m$.  Let $\{f_1, \dots, f_M\}$ be an orthonormal basis of $\mathcal{H}(K_2)$. Then by \cite[Proposition $2.18$]{am01}, we have $K_2(z,w) = \sum_{i=1}^M f_i(z) \overline{f_i(w)}.$

 Fix $\tau \in \mathbb{T} \setminus E_{\Theta}$. Then $\Theta_{\tau}$ is a one-variable rational inner function and thus, is a Blaschke product with $\deg \Theta_{\tau} \le m.$  Further, as $p$ has no zeros on $\mathbb{D}\times \mathbb{T}$,  one can show that $\deg \tilde{p}(\cdot, \tau) =m.$  Since $p(\cdot, \tau)$ also has no zeros on $\mathbb{T}$, no polynomials cancel in the fraction $\Theta_{\tau}=\frac{\tilde{p}(\cdot, \tau)}{p(\cdot, \tau)}$. This implies $\deg \Theta_{\tau} =m$ and $\dim \mathcal{K}_{\Theta_{\tau}} = m$. Now, letting $z_2, w_2 \rightarrow \tau$ in \eqref{eqn:ad} and dividing by $1-z_1\overline{w_1}$ gives
\[ \frac{1-\Theta_{\tau}(z_1) \overline{\Theta_{\tau}(w_1)}}{1-z_1 \overline{w_1}} = \sum_{i=1}^M f_i(z_1, \tau) \overline{f_i(w_1, \tau)}.\]
Thus, the set $\{ f_1(\cdot, \tau), \dots, f_M(\cdot, \tau)\}$ spans $\mathcal{K}_{\Theta_{\tau}}$ and so the restriction map $\mathcal{J}_{\tau}$ is well defined (i.e.~maps $\mathcal{H}(K_2)$ into $\mathcal{K}_{\Theta_{\tau}}$) and is surjective. 

To show that each $\mathcal{J}_{\tau}$ is an isometry, fix $f,g \in \mathcal{H}(K_2)$ and for $z_2 \in \mathbb{T}$, define
\[ F_{f,g}(z_2) :=  \int_{\mathbb{T}} f(z_1, z_2) \overline{g(z_1, z_2)} d\sigma(z_1) = \left \langle f(\cdot, z_2), g(\cdot, z_2)\right \rangle_{\mathcal{K}_{\Theta_{z_2}}},\]
where $d\sigma(z_1)$ is normalized Lebesgue measure on $\mathbb{T}$ and the last equality holds for $z_2 \in \mathbb{T} \setminus E_{\Theta}.$ An application of H\"older's inequality immediately implies that $F_{f,g} \in L^1(\mathbb{T})$. Furthermore, our assumptions imply that $\mathcal{H}(K_2) \perp z_2 \mathcal{H}(K_2).$ From this we can conclude that $f \perp z_2^j g$ in $\mathcal{S}_2$, and hence in $H^2(\mathbb{D}^2)$, for all $j \in \mathbb{Z} \setminus \{0\}$. Then the Fourier coefficients of $F_{f,g}$ can be computed as follows:
\[  \widehat{F_{f,g}}(j) =  \int_{\mathbb{T}} z_2^{-j} F_{f,g}(z_2) d\sigma(z_2) =\int_{\mathbb{T}^2} z_2^{-j} f(z) \overline{g(z)} d\sigma(z_1)d\sigma(z_2)=0\]
for $j \in \mathbb{Z} \setminus \{0\}$. Then basic Fourier analysis (for example, Corollary 8.45 in \cite{folland}) implies that 
\[ F_{f,g}(z_2)  =  \widehat{F_{f,g}}(0)  = \left \langle f,g \right \rangle_{\mathcal{H}(K_2)}\ \text{ for a.e.~} z_2 \in \mathbb{T}.\]
But, the formula for $F_{f,g}$ implies that it is continuous on $\mathbb{T}\setminus E_{\Theta}$ and so for $z_2 \in \mathbb{T}\setminus E_{\Theta}$,
\[  \left \langle f(\cdot, z_2), g(\cdot, z_2)\right \rangle_{\mathcal{K}_{\Theta_{z_2}}} = F_{f,g}(z_2)= \left \langle f,g \right \rangle_{\mathcal{H}(K_2)}.\]
This implies $\mathcal{J}_{\tau}$ is an isometry for $\tau \in \mathbb{T} \setminus E_{\Theta}$. Since it is also surjective, $\mathcal{J}_{\tau}$ is unitary and so
\[ \dim \mathcal{H}(K_2) = \dim \mathcal{K}_{\Theta_{\tau}}  = m,\]
completing the proof. \end{proof}

\begin{remark} \label{rem:kernel} Let $\Theta = \frac{\tilde{p}}{p}$ be rational inner with $\deg \Theta = (m,n)$ and let $\mathcal{S}_2$ be as in \eqref{eqn:sub}. Then Theorems \ref{thm:dim} and \ref{thm:rational} can be used to deduce information about both the functions in $\mathcal{S}_2$ and the inner product of $\mathcal{S}_2$. As mentioned earlier,  we let $H^2_2(\mathbb{D})$ denote the one-variable Hardy space with independent variable $z_2.$

First, as in \eqref{eqn:kernel}, let $K_2$ be the reproducing kernel satisfying $\mathcal{H}(K_2) = \mathcal{S}_2 \ominus z_2 \mathcal{S}_2.$ By Theorems \ref{thm:dim} and \ref{thm:rational}, there is an orthonormal basis $\{ \frac{q_1}{p}, \dots, \frac{q_m}{p}\} $ of $\mathcal{H}(K_2)$ with $\deg q_i \le (m-1,n)$ for $i=1,\dots, m.$  Then, since $\frac{q_i}{p} \perp \frac{q_j}{p} z^k$ for all $i \ne j$ and $k\in \mathbb{Z}$, one can show
\begin{equation} \label{eqn:S2decomp} \mathcal{S}_2 = \mathcal{H}\left( \frac{\sum_{i=1}^m \frac{q_i(z)}{p(z)} \frac{\overline{q_i(w)}}{\overline{p(w)}}}{1-z_2\bar{w}_2}\right) = \bigoplus_{i=1}^m \mathcal{H}\left( \frac{ \frac{q_i(z)}{p(z)} \frac{\overline{q_i(w)}}{\overline{p(w)}}}{1-z_2\bar{w}_2}\right),
\end{equation}
where the last term indicates an orthogonal decomposition of $\mathcal{S}_2$ into $m$ subspaces. We also claim that each subspace
\[ \mathcal{S}_2^i := \mathcal{H}\left( \frac{ \frac{q_i(z)}{p(z)} \frac{\overline{q_i(w)}}{\overline{p(w)}}}{1-z_2\bar{w}_2}\right)\]
 is precisely the set of functions $\frac{q_i}{p}H_2^2(\mathbb{D})$ and for each pair of functions $\frac{q_i}{p}f_i,  \frac{q_i}{p}g_i \in \mathcal{S}_2^i$, 
 \begin{equation} \label{eqn:inner} \left \langle \frac{q_i}{p} f_i, \frac{q_i}{p} g_i \right \rangle_{\mathcal{S}_2^i} = \left \langle f_i, g_i \right \rangle_{H_2^2(\mathbb{D})}.\end{equation}
 One can prove this claim by defining the above inner product on the set $\frac{q_i}{p}H_2^2(\mathbb{D})$. A straightforward computation shows that this turns  $\frac{q_i}{p}H_2^2(\mathbb{D})$ into a reproducing kernel Hilbert space with reproducing kernel $\frac{q_i(z)}{p(z)} \frac{\overline{q_i(w)}}{\overline{p_i(w)}} \frac{1}{1-z_2\bar{w}_2}$. By the uniqueness of reproducing kernels, the set $\frac{q_i}{p} H_2^2(\mathbb{D})$ with the proposed inner product is exactly $\mathcal{S}_2^i.$

Then, we can define a linear map $\mathcal{U}: H^2_2(\mathbb{D})^m \rightarrow \mathcal{S}_2$ by
\begin{equation} \label{eqn:unitary} \mathcal{U} \vec{f} := \sum_{i=1}^m \frac{q_i}{p} f_i,\qquad \text{ for } \vec{f} = (f_1, \dots,f_m )\in H^2_2(\mathbb{D})^m.\end{equation} 
We will show that this map is actually unitary. First, observe that this map is well defined and surjective since \eqref{eqn:S2decomp} and the above characterization of the subspaces $\mathcal{S}_2^i$ imply that $\mathcal{S}_2$ is composed precisely of functions of the form $\sum_{i=1}^m \frac{q_i}{p} f_i$, where each $f_i \in H_2^2(\mathbb{D})$. Moreover, as \eqref{eqn:S2decomp} is an orthogonal decomposition and \eqref{eqn:inner} gives the inner product on each $\mathcal{S}_2^i$, we can conclude that for all $\vec{f}, \vec{g} \in H^2_2(\mathbb{D})^m$, 
\[
  \left \langle \mathcal{U} \vec{f} , \mathcal{U} \vec{g } \right \rangle_{\mathcal{S}_2}  = \sum_{i=1}^m \left \langle \frac{q_i}{p}f_i, \frac{q_i}{p}g_i \right \rangle_{\mathcal{S}_2^i}
 = \sum_{i=1}^m \left \langle f_i, g_i \right \rangle_{H^2_2(\mathbb{D})} = \left \langle \vec{f}, \vec{g} \right \rangle_{H_2^2(\mathbb{D})^m}.\] 
Thus, $\mathcal{U}$ is unitary as desired.
\end{remark}

\section{The Structure and Numerical Range of $S^1_{\Theta}$  }\label{sec:structure}

Let $\Theta$ be rational inner and write $\mathcal{K}_{\Theta} = \mathcal{S}_1 \oplus \mathcal{S}_2,$ for subspaces $\mathcal{S}_1$, $\mathcal{S}_2$ that are respectively $z_1$- and $z_2$-invariant. As the following lemma shows, the numerical range of $P_{\mathcal{S}_1}  \widetilde{S}^1_{\Theta}|_{\mathcal{S}_1}$ is not particularly interesting.

\begin{lemma} \label{lem:range} Let $\Theta = \frac{\tilde{p}}{p}$ be rational inner of degree $(m,n)$ and let $\mathcal{S}_1$  be a $z_1$-invariant subspace of $\mathcal{K}_{\Theta}$ as in \eqref{eqn:sub}.
\begin{itemize}
\item[a.] If $n=0$, then Clos $\mathcal{W}(P_{\mathcal{S}_1}  \widetilde{S}^1_{\Theta}|_{\mathcal{S}_1})= \{0\}.$ 
\item[b.] If $n > 0$, then Clos $\mathcal{W}(P_{\mathcal{S}_1} \widetilde{S}^1_{\Theta}|_{\mathcal{S}_1})= \overline{\mathbb{D}}.$ 
\end{itemize}
\end{lemma}
\begin{proof} Let $K_1$ be as in \eqref{eqn:kernel}, i.e.~the reproducing kernel satisfying $\mathcal{H}(K_1) = \mathcal{S}_1 \ominus z_1 \mathcal{S}_1.$ Then 
\[ \mathcal{S}_1 = \bigoplus_{k=0}^{\infty} z_1^k \mathcal{H}(K_1).\]
 If $n=0$, then Theorem \ref{thm:rational} implies that $\dim \mathcal{H}(K_1) =0$, so $\mathcal{S}_1 = \{0\}.$ It follows immediately that  Clos $\mathcal{W}(P_{\mathcal{S}_1}  \widetilde{S}^1_{\Theta}|_{\mathcal{S}_1})= \{0\}.$ 

Now assume $n>0$. Then by Theorems~\ref{thm:dim} and \ref{thm:rational}, we can find an orthonormal basis $\{ \frac{r_1}{p}, \dots, \frac{r_n}{p}\}$ of $\mathcal{H}(K_1)$ with each $r_i$ a polynomial. Define
\[ Z_{K_1} := \{ w_1 \in \mathbb{D} : r_i(w_1, w_2) =0 \ \text{ for all } w_2 \in \mathbb{D} \text{ and } 1 \le i \le n\}.\]
If $w_1 \in Z_{K_1}$, then each $r_i(w_1, \cdot) \equiv 0$ on $\mathbb{D}$. Thus, $ r_i(w_1, \cdot) \equiv 0$ on $\mathbb{C}$. This implies $r_i$ vanishes on the zero set of $z_1-w_1.$  Since $z_1-w_1$ is irreducible, Hilbert's Nullstellensatz implies that $z_1 - w_1$ divides each $r_i$ and as the $r_i$ are polynomials, this implies that $Z_{K_1}$ is a finite set. Observe that 
\[ \widehat{K}_1(z,w) := \frac{K_1(z,w)}{1-z_1\bar{w}_1} = \sum_{i=1}^n \frac{\frac{r_i(z)}{p(z)}\frac{\overline{r_i(w)}}{\overline{p(w)}}}{1-z_1\bar{w}_1}\]
is the reproducing kernel for $\mathcal{S}_1$. Fix $w_1 \in \mathbb{D} \setminus Z_{K_1}$ and choose $w_2 \in \mathbb{D}$ so that at least one $r_i(w_1, w_2) \ne 0.$ Then setting $w=(w_1, w_2)$, we have $ \| \widehat{K}_1(\cdot, w)\|^2_{\mathcal{S}_1}  = \widehat{K}_1(w,w) \ne 0$ and since $\mathcal{S}_1$ is $z_1$-invariant,
\[  w_1 \left \| \widehat{K}_1(\cdot, w) \right\|^2_{\mathcal{S}_1} = w_1\widehat{K}_1(w,w) = \left \langle M_{z_1}\widehat{K}_1(\cdot, w), \widehat{K}_1(\cdot, w) \right \rangle_{\mathcal{S}_1} =  \left \langle  \widetilde{S}^1_{\Theta} \widehat{K}_1(\cdot, w), \widehat{K}_1(\cdot, w) \right \rangle_{\mathcal{S}_1}. \]
 Since $\| \widehat{K}_1(\cdot, w)\|^2_{\mathcal{S}_1}  \ne 0$, we can divide both sides of the above equation by it and conclude that the point $w_1 \in \mathcal{W}(P_{\mathcal{S}_1}  \widetilde{S}^1_{\Theta}|_{\mathcal{S}_1}).$ Since this works for all $w_1 \in \mathbb{D} \setminus Z_{K_1}$ and $Z_{K_1}$ is finite,
 \[\overline{\mathbb{D}} \subseteq \text{Clos }\mathcal{W}(P_{\mathcal{S}_1}  \widetilde{S}^1_{\Theta}|_{\mathcal{S}_1}).\]
 The other containment follows immediately because $ \widetilde{S}^1_{\Theta}$ is a contraction. \end{proof}

By Lemma \ref{lem:range}, the interesting behavior of $ \widetilde{S}^1_{\Theta}$ occurs on the subspace $\mathcal{S}_2$.  Because of this, as mentioned earlier, we primarily study this alternate compression of the shift
\[ S_{\Theta}^1 : = P_{\mathcal{S}_2} \widetilde{S}^1_{\Theta}|_{\mathcal{S}_2} =  P_{\mathcal{S}_2} M_{z_1} |_{\mathcal{S}_2}.\]
In the following result, we show that $S^1_{\Theta}$ is unitarily equivalent to a simple $z_2$-matrix-valued Toeplitz operator, as defined in \eqref{eqn:Toeplitz}.

\begin{theorem} \label{thm:unitary} Let $\Theta = \frac{\tilde{p}}{p}$ be rational inner of degree $(m,n)$ and let $\mathcal{S}_2$  be as in \eqref{eqn:sub}. Then there exists an $m\times m$  matrix-valued function $M_{\Theta}$, with entries that are rational functions of $\bar{z}_2$ and continuous on $\overline{\mathbb{D}},$ such that
\begin{equation} \label{eqn:unitary1}  S^1_{\Theta} = \mathcal{U} \ T_{M_{\Theta}} \ \mathcal{U}^*, \end{equation}
where $\mathcal{U}: H^2_2(\mathbb{D})^m \rightarrow \mathcal{S}_{2}$ is the unitary operator defined in \eqref{eqn:unitary}.
\end{theorem}

\begin{proof} Throughout this proof, we use the notation defined and explained in Remark \ref{rem:kernel}. Recall that $\{ \frac{q_1}{p}, \dots, \frac{q_m}{p}\} $ denotes the previously-obtained orthonormal basis of $\mathcal{H}(K_2):=  \mathcal{S}_2 \ominus z_2 \mathcal{S}_2.$

By Proposition $3.4$ in \cite{b13}, $\mathcal{S}_2$ is invariant under the backward shift operator $S_{\Theta}^{1*}=M_{z_1}^*|_{\mathcal{S}_2}$. This means that there are one-variable functions $h_{1j}, \dots, h_{mj} \in H_2^2(\mathbb{D})$ such that
\begin{equation} \label{eqn:bwsform1}  M_{z_1}^*\left( \frac{q_j}{p} \right)= \frac{q_1}{p} h_{1j}  + \dots + \frac{q_m}{p} h_{mj}, \quad \text{ for } j=1, \dots, m.\end{equation} Define the $m\times m$ matrix-valued function $H$ by
\begin{equation} \label{eqn:H} H:= \begin{bmatrix} h_{11} & \cdots & h_{1m} \\
\vdots & \ddots & \vdots \\
h_{m1} & \cdots & h_{mm}
\end{bmatrix},
\end{equation}
and define the matrix-valued function $M_{\Theta}$ by
\begin{equation} \label{eqn:M} M_{\Theta} := H^*. \end{equation}
 To establish the properties of $M_{\Theta}$, we will show that $H$ has entries that are rational in $z_2$ and continuous on $\overline{\mathbb{D}}.$ First rewrite the terms in \eqref{eqn:bwsform1} as
\[ \left(M_{z_1}^*\frac{q_j}{p} \right)(z) = \frac{Q_j(z)}{p(z) p(0,z_2)} \  \text{ and } \ Q_{j}(z)  = \sum_{k=0}^{m-1} Q_{jk}(z_2)z_1^k\]
for a polynomial $Q_j$ and write 
\[ q_i(z) = \sum_{k=0}^{m-1} q_{ik}(z_2)z_1^k, \qquad \text{ for } i=1, \dots, m.\]
Then by canceling the $p$ from each denominator from \eqref{eqn:bwsform1} and looking at the coefficients in front of each $z_1^k$ separately, \eqref{eqn:bwsform1} can be rewritten as 
\[ \begin{bmatrix} q_{10}(z_2) & \dots & q_{m0}(z_2) \\
\vdots & \ddots & \vdots \\
q_{1(m-1)}(z_2) & \dots & q_{m(m-1)}(z_2) 
\end{bmatrix} 
\begin{bmatrix} 
h_{1j}(z_2) \\
\vdots \\
h_{mj}(z_2) 
\end{bmatrix} 
= \begin{bmatrix}
\frac{Q_{j0}(z_2)}{p(0,z_2)} \\
\vdots \\
\frac{Q_{j(m-1)}(z_2)}{p(0,z_2)} 
\end{bmatrix},
\]
for $z_2 \in \mathbb{D}$ and $j=1, \dots, m.$
Let $A_j$ denote the $m\times m$ matrix function in the above equation. Since $\det A_j$ is a one-variable polynomial, it is either identically zero or has finitely many zeros. First assume $\det A_j\equiv 0$, so that clearly $\det A_j(\tau) = 0$ for each $\tau \in \mathbb{T}$. This implies that for each fixed $\tau \in \mathbb{T}$, one $q_k(\cdot, \tau)$ can be written as a linear combination of the other $q_i(\cdot, \tau).$ However by Theorem \ref{thm:rational}, for $\tau\in \mathbb{T} \setminus E_{\Theta}$, the set 
\[ \left\{\frac{q_1}{p}(\cdot,\tau), \dots , \frac{q_m}{p}(\cdot,\tau)\right \}\]
 is a basis for the $m$-dimensional set $\mathcal{K}_{\Theta_\tau}$. Thus the set must be linearly independent, a contradiction.

Hence, $\det A_j \not \equiv 0.$  Thus, the matrix $A_j(z_2)$ is invertible except at (at most) a finite number of points $z_2 \in \mathbb{D}$ and so we can solve for each column of $H$ as 
\[\begin{bmatrix} 
h_{1j}(z_2) \\
\vdots \\
h_{mj}(z_2) 
\end{bmatrix}  =  \begin{bmatrix} q_{10}(z_2) & \dots & q_{m0}(z_2) \\
\vdots & \ddots & \vdots \\
q_{1(m-1)}(z_2) & \dots & q_{m(m-1)}(z_2) 
\end{bmatrix}^{-1} 
\begin{bmatrix}
\frac{r_{j0}(z_2)}{p(0,z_2)} \\
\vdots \\
\frac{r_{j(m-1)}(z_2)}{p(0,z_2)} 
\end{bmatrix}.
\]
This shows that the entries of $H$ are rational functions in $z_2$ and so by \eqref{eqn:M}, the entries of $M_{\Theta}$ are rational in $\bar{z}_2$.

Since the entries of $H$ are also in $H_2^2(\mathbb{D})$, we claim that they cannot have any singularities in $\overline{\mathbb{D}}$. That there are no singularities in $\mathbb{D}$ should be clear. To see that there are no singularities on $\mathbb{T}$, proceed by contradiction and assume that some $h_{ij}$ has a singularity at a $\tau \in \mathbb{T}$. Then, after writing $h_{ij}$ as a ratio of one-variable polynomials with no common factors, the denominator of $h_{ij}$ vanishes at $\tau$ but the numerator does not. 
By the reproducing property of $H_2^2(\mathbb{D})$, we know that for each $z_2 \in \mathbb{D}$, 
\[ | h_{ij}(z_2) | = \left | \left \langle h_{ij}, \frac{1}{1-\cdot \bar{z_2}} \right \rangle_{H_2^2(\mathbb{D})} \right|  \le \| h_{ij} \|_{H_2^2(\mathbb{D})} \frac{1}{ \sqrt{1- |z_2|^2}}.\]
But since $h_{ij}$ has a singularity at $\tau$, there is a sequence $\{z_{2,n}\} \rightarrow \tau$ and positive constant $C$ such that $|h_{ij}(z_{2,n})| \ge C \frac{1}{1-|z_{2,n}|}$ for each $n$, a contradiction. Thus $H$, and hence $M_{\Theta},$ has entries continuous on $\overline{\mathbb{D}}.$ 

Now we establish \eqref{eqn:unitary1}. Fix $f,g \in \mathcal{S}_2$. Then by Remark \ref{rem:kernel}, there exist vector-valued functions $\vec{f}=(f_1, \dots, f_m) , \vec{g}=(g_1, \dots, g_m) \in  H^2_2(\mathbb{D})^m$ such that 
\[ f =\sum_{i=1}^m \frac{q_i}{p} f_i= \mathcal{U} \ \vec{f} \ \ \text{ and } \ \ g =\sum_{i=1}^m \frac{q_i}{p} g_i = \mathcal{U} \ \vec{g},\]
so $\vec{f} = \mathcal{U}^*f$ and $\vec{g} = \mathcal{U}^* g.$ Then using the inner product formulas from  Remark \ref{rem:kernel}, we can compute
\[
\begin{aligned} 
\left \langle S^1_{\Theta} f, g  \right \rangle_{\mathcal{S}_2}
 &= \left \langle f, M_{z_1}^* g \right \rangle_{\mathcal{S}_2} \\
& = \left \langle  \sum_{i=1}^m \frac{q_i}{p} f_i, \sum_{j=1}^m M^*_{{z}_1} \left( \frac{q_j}{p} \right) g_j \right \rangle_{\mathcal{S}_2} \\
& = \sum_{i,j=1}^m \left \langle \frac{q_i}{p} f_i,   \frac{q_i}{p} h_{ij} g_j \right \rangle_{\mathcal{S}_2}\\
& = \sum_{i,j=1}^m \left \langle  f_i,  h_{ij} g_j \right \rangle_{H_2^2(\mathbb{D})}\\
& = \left \langle \vec{f}, T_H \vec{g} \right \rangle_{H_2^2(\mathbb{D})^m} \\
& = \left \langle T_{M_{\Theta}} \vec{f}, \vec{g} \right \rangle_{H_2^2(\mathbb{D})^m}\\
& =  \left \langle \ \mathcal{U} \ T_{M_{\Theta}} \ \mathcal{U}^* f, g\right \rangle_{\mathcal{S}_2},
\end{aligned}
\]
where $T_H$ is the $z_2$-matrix-valued Toeplitz operator with symbol $H$. Since $f, g \in \mathcal{S}_2$ were arbitrary, this immediately gives \eqref{eqn:unitary1}. \end{proof}

\begin{example} Before proceeding, observe that Theorem \ref{thm:unitary} generalizes the matrix from \eqref{eqn:onevar}. Specifically, let $\Theta = \frac{\tilde{p}}{p}$ be a rational inner function with $\deg \Theta = (m,0)$, so $\Theta$ is a finite Blaschke product of degree $m$. Then the associated two-variable model space is
\[ \mathcal{K}_{\Theta} = \mathcal{H} \left( \frac{1- \Theta(z_1)\overline{\Theta(w_1)}}{(1-z_1\overline{w}_1)(1-z_2\overline{w}_2)}\right),\]
which is $z_2$-invariant. Thus, we can set $\mathcal{S}_2 = \mathcal{K}_{\Theta}$ and $\mathcal{S}_1 = \{0\}$. One can actually show that this is the only choice of $\mathcal{S}_1$ and $\mathcal{S}_2.$ Then
\[ \mathcal{H}(K_2) := \mathcal{S}_2 \ominus z_2 \mathcal{S}_2 = \mathcal{H} \left( \frac{1-\Theta(z_1)\overline{\Theta(w_1)}}{1-z_1\overline{w}_1}\right)\]
is the one-variable model space associated to $\Theta$ with independent variable $z_1$. It follows immediately that the one-variable Takenaka-Malmquist-Walsh basis $\{f_1, \dots, f_m\}$ is an orthonormal basis for $\mathcal{H}(K_2)$ and each $f_i = \frac{q_i}{p}$ for some one-variable polynomial $q_i$ with $\deg q_i \le m-1.$ Because the one-variable model space (with independent variable $z_1$) is also invariant under the backward shift $M_{z_1}^*$, we can conclude that the unique $h_{ij}$ from \eqref{eqn:bwsform1} are constants. Then since $\mathcal{H}(K_2)$ is a subspace of $\mathcal{K}_{\Theta}$, we can use \eqref{eqn:bwsform1}-\eqref{eqn:M} to conclude
\[ \left( {M_{\Theta}}\right)_{ij} = \overline{H_{ji}} =  \overline{ \left \langle M_{z_1}^*\frac{q_i}{p}, \frac{q_j}{p} \right \rangle_{\mathcal{K}_{\Theta}}} = \overline{ \left \langle M_{z_1}^*\frac{q_i}{p}, \frac{q_j}{p} \right \rangle_{\mathcal{H}(K_2)}}  =  \left \langle  P_{\mathcal{H}(K_2)} M_{z_1} f_j ,f_i \right \rangle_{\mathcal{H}(K_2)},\]
which is a constant matrix agreeing with the matrix  from \eqref{eqn:onevar}.
\end{example}

As a corollary of Theorem \ref{thm:unitary}, we can characterize the numerical range of $ S^1_{\Theta},$ denoted by $\mathcal{W}( S^1_{\Theta})$.

\begin{corollary} \label{thm:nr} Let $\Theta = \frac{\tilde{p}}{p}$ be rational inner of degree $(m,n)$, let $\mathcal{S}_2$  be as in \eqref{eqn:sub}, and let $M_{\Theta}$ be as in Theorem \ref{thm:unitary}. Then
\begin{equation} \label{eqn:nr} \text{Clos}\left( \mathcal{W}\left ( S^1_{\Theta} \right)\right) =  \text{Conv} \Big( \ \bigcup_{\tau \in \mathbb{T}} \mathcal{W}\left(M_{\Theta}(\tau) \right) \ \Big).\end{equation}
\end{corollary}

\begin{proof}  By Theorem \ref{thm:unitary}, the operator $ S^1_{\Theta}$ has the same numerical range as the $z_2$-matrix-valued Toeplitz operator $T_{M_\Theta}: H_2^2(\mathbb{D})^m \rightarrow  H_2^2(\mathbb{D})^m$. By \cite[Theorem 1]{SpitBeb}, the closure of the numerical range $\mathcal{W}(T_{M_{\Theta}})$ is equal to
$$\mbox{Conv} \left \{\mathcal{W}(A): A \in \mathcal{R}(M_{\Theta})\right\},$$ 
where $\mathcal{R}(M_{\Theta} )$ is the essential range of $M_{\Theta}$ as a function on $\mathbb{T}.$ It is easy to see that this set is closed and so, we do not need to take its closure. Since $M_{\Theta}$ is continuous on $\mathbb{T}$, its essential range will equal its range, i.e.
\[  \mbox{Conv} \left\{\mathcal{W}(A): A \in \mathcal{R}(M_{\Theta}) \right\} =  \text{Conv} \Big( \ \bigcup_{\tau \in \mathbb{T}} \mathcal{W}\left(M_{\Theta}(\tau) \right) \ \Big), \]
proving \eqref{eqn:nr}.
\end{proof}

One can also consider the family of one-variable functions $\{ \Theta_{\tau} = \Theta(\cdot, \tau): \tau \in \mathbb{T} \setminus E_{\Theta} \},$ where $E_{\Theta}$ is the exceptional set defined in \eqref{eqn:exceptional}. For each $\tau \in \mathbb{T}\setminus E_{\Theta}$, let $S_{\Theta_{\tau}}$  denote the compression of the shift on $\mathcal{K}_{\Theta_{\tau}},$ the one-variable model space associated to $\Theta_{\tau}.$ It turns out that the numerical ranges $\mathcal{W}(S_{\Theta_{\tau}})$ are closely related to $\mathcal{W}( S^1_{\Theta}).$

\begin{theorem} \label{thm:generalnr} Let $\Theta = \frac{\tilde{p}}{p}$ be rational inner of degree $(m,n)$, let $E_{\Theta}$ be the exceptional set from \eqref{eqn:exceptional} and let $\mathcal{S}_2$  be as in \eqref{eqn:sub}. Then
\begin{equation} \label{eqn:onevariable}\text{Clos}\left( \mathcal{W} \left( S^1_{\Theta}\right) \right) = \text{Clos} \Big( \text{Conv}\Big( \bigcup_{\tau \in \mathbb{T} \setminus E_{\Theta}} \mathcal{W}(S_{\Theta_{\tau}})\Big)\Big).\end{equation}
\end{theorem}

\begin{proof} This proof will use the same notation as the proof of Theorem \ref{thm:unitary}. First fix $\tau \in \mathbb{T} \setminus E_{\Theta}$. By Theorem \ref{thm:rational}, the set
\[ \Big\{ \frac{q_1}{p}(\cdot, \tau), \dots, \frac{q_m}{p}(\cdot, \tau) \Big\}\]
is an orthonormal basis for $\mathcal{K}_{\Theta_{\tau}}$ with independent variable $z_1.$ Consider  \eqref{eqn:bwsform1}. As all involved functions are rational with no singularities on $\overline{\mathbb{D}} \times \left(\overline{\mathbb{D}} \setminus E_{\Theta} \right)$ and the backward shift operator $S_{\Theta}^{1*}=M^*_{z_1}|_{\mathcal{S}_2}$ treats $z_2$ like a constant, we can extend this formula to the functions $\frac{q_i}{p}(\cdot, \tau).$
Specifically, \[  M_{z_1}^* \left( \frac{q_j}{p}(\cdot, \tau) \right) = \left( M_{z_1}^* \frac{q_j}{p}\right)(\cdot, \tau) =  \frac{q_1}{p}(\cdot, \tau) h_{1j}(\tau)  + \dots + \frac{q_m}{p}(\cdot, \tau) h_{mj}(\tau), 
\]
for $j=1, \dots, m.$ Now, we use arguments similar to those in the proof of Theorem \ref{thm:unitary} to show $\mathcal{W}(S_{\Theta_{\tau}}) = \mathcal{W}(M_{\Theta}(\tau)).$ 
Specifically, fix $f \in \mathcal{K}_{\Theta_{\tau}}$. Then there exist unique constants
$a_1, \dots, a_m \in \mathbb{C}$ such that
\[ f= \sum_{i=1}^m a_i \frac{q_i}{p}(\cdot,\tau).\]
Moreover,  $\| f\|^2_{\mathcal{K}_{\Theta_{\tau}}} =1$ if and only if   $\sum_{i=1}^m |a_i|^2=1,$ i.e.~exactly when
$\vec{a}:= (a_1, \dots, a_m) \in \mathbb{C}^m$ has norm one. Then,
\[
\begin{aligned} 
\left \langle S_{\Theta_{\tau}} f,f \right \rangle_{\mathcal{K}_{\Theta_{\tau}}} &= \left \langle f,  M_{z_1}^* f \right \rangle_{\mathcal{K}_{\Theta_{\tau}}} \\
& = \left \langle  \sum_{i=1}^m a_i \frac{q_i}{p}(\cdot, \tau), \sum_{j=1}^m a_j M_{z_1}^*\left( \frac{q_j}{p}(\cdot, \tau) \right)  \right \rangle_{\mathcal{K}_{\Theta_{\tau}}} \\
& = \sum_{i,j=1}^m \left \langle a_i \frac{q_i}{p}(\cdot, \tau),   \frac{q_i}{p}(\cdot, \tau) a_j h_{ij}(\tau)  \right \rangle_{\mathcal{K}_{\Theta_{\tau}}}\\
& = \sum_{i,j=1}^m \left \langle a_i,  a_j  h_{ij}(\tau) \right \rangle_{\mathbb{C}} \\
& = \left \langle \vec{a}, H(\tau)\vec{a}  \right \rangle_{\mathbb{C}^m} \\
& = \left \langle M_{\Theta}(\tau) \vec{a}, \vec{a} \right \rangle_{\mathbb{C}^m}, 
\end{aligned}
\]
where we used the definitions of $H$ and $M_{\Theta}$ from \eqref{eqn:H} and \eqref{eqn:M}. This sequence of equalities
proves that $\mathcal{W}(S_{\Theta_{\tau}}) = \mathcal{W}(M_{\Theta}(\tau)).$ Thus, we have 

\[  
\begin{aligned}
\text{Clos} \Big( \text{Conv}\Big( \bigcup_{\tau \in \mathbb{T} \setminus E_{\Theta}} \mathcal{W}(S_{\Theta_{\tau}})\Big)\Big) &=  \text{Clos} \Big( \text{Conv}\Big( \bigcup_{\tau \in \mathbb{T} \setminus E_{\Theta}} \mathcal{W}(M_{\Theta}(\tau))\Big)\Big)\\
& =  \text{Clos} \Big( \text{Conv}\Big( \bigcup_{\tau \in \mathbb{T}} \mathcal{W}(M_{\Theta}(\tau))\Big)\Big) \\
& = \text{Clos}\left( \mathcal{W}( S^1_{\Theta}) \right),
\end{aligned}
\] 
where we used Corollary \ref{thm:nr} and the fact that $M_{\Theta}$ is continuous on $\mathbb{T}$.
\end{proof}

If $\Theta = \frac{\tilde{p}}{p}$ is rational inner of degree $(m,n)$, then there are typically many ways to decompose $\mathcal{K}_{\Theta}$ into shift invariant subspaces $\mathcal{S}_1$ and $\mathcal{S}_2$. Indeed, according to Corollary 13.6 in \cite{k14}, if $\deg p = \deg \tilde{p}$,  there is a unique such decomposition if and only if $\tilde{p}$ and $p$ have $2mn$ common zeros (including intersection multiplicity) on $\mathbb{T}^2.$ 
Nevertheless, Theorem \ref{thm:generalnr} allows us to show that $\mathcal{W}( S^1_{\Theta})$ does not depend on the decomposition chosen.

\begin{corollary} \label{cor:S2} Let $\Theta = \frac{\tilde{p}}{p}$ be rational inner of degree $(m,n)$.  Let 
\[ \mathcal{K}_{\Theta} = \mathcal{S}_1 \oplus \mathcal{S}_2 = \widetilde{\mathcal{S}}_1 \oplus \widetilde{\mathcal{S}}_2\]
where both $\mathcal{S}_j, \widetilde{\mathcal{S}}_j$ are $z_j$-invariant subspaces for $j=1,2$. Then
\[ \text{Clos}  \left( \mathcal{W}\left( P_{\mathcal{S}_2} M_{z_1} |_{\mathcal{S}_2} \right)\right)= \text{Clos}  \left( \mathcal{W}\left(P_{\widetilde{\mathcal{S}}_2}  M_{z_1}|_{\widetilde{\mathcal{S}}_2} \right)\right).\]
\end{corollary}

\begin{proof} By Theorem \ref{thm:generalnr},
 \[ 
 \text{Clos}  \left( \mathcal{W}\left(P_{\mathcal{S}_2} M_{z_1} |_{\mathcal{S}_2} \right)\right)=
  \text{Clos} \Big( \text{Conv}\Big( \bigcup_{\tau \in \mathbb{T} \setminus E_{\Theta}} \mathcal{W}(S_{\Theta_{\tau}})\Big)\Big) =
  \text{Clos}  \left( \mathcal{W}\left(P_{\widetilde{\mathcal{S}}_2} M_{z_1} |_{\widetilde{\mathcal{S}}_2} \right)\right),
  \]
  as desired. \end{proof}
  
  Theorem \ref{thm:generalnr} is particularly useful because the compressions of the shift on one-variable model spaces are well studied. Specifically, let $B$ be a degree $m$ Blaschke product with zeros $\alpha_1, \dots, \alpha_m$ and let $S_B$ denote the compression of the shift on $\mathcal{K}_B.$ Then, as mentioned in the introduction, one matrix of $S_B$ is given by \eqref{eqn:onevar}. Using this formula, it is easy to deduce that the zeros $\alpha_1, \dots, \alpha_m$ are all in $\mathcal{W}(S_B)$. We will use this to establish the following result:
  
\begin{theorem}\label{thm:numericalradius} Let $\Theta = \frac{\tilde{p}}{p}$ be rational inner of degree $(m,n)$ and let $\mathcal{S}_2$  be as in \eqref{eqn:sub}. Then the numerical radius $w\big( S^1_{\Theta} \big) =1$ if and only if $\Theta$ has a singularity on $\mathbb{T}^2$.
\end{theorem} 
\begin{proof} 
($\Rightarrow$) Assume $w\big( S^1_{\Theta} \big) =1$. Then there exists a sequence $\{\lambda_n\} \subseteq \mathcal{W}( S^1_{\Theta})$ such that $|\lambda_n| \rightarrow 1.$ Since $\{\lambda_n\}$ is bounded, it has a subsequence converging to some $\lambda \in \mathbb{T}$. Thus, $\lambda \in \text{Clos}(\mathcal{W}( S^1_{\Theta})).$ By Corollary \ref{thm:nr}, 
\[ \lambda  \in   \text{Conv} \Big( \ \bigcup_{\tau \in \mathbb{T}} \mathcal{W}\left(M_{\Theta}(\tau) \right) \ \Big).\]
Again by Corollary \ref{thm:nr}, as $S^1_{\Theta}$ is a contraction, every $\alpha \in \bigcup_{\tau \in \mathbb{T}} \mathcal{W}\left(M_{\Theta}(\tau)\right)$ satisfies $|\alpha| \le 1$. Since $|\lambda|=1$, we can conclude that there is some $\tilde{\tau} \in \mathbb{T}$ and some $\tilde{\lambda} \in  \mathcal{W}\left(M_{\Theta}(\tilde{\tau})\right)$ such that $|\tilde{\lambda}|=1.$  

Now by way of contradiction, assume $\Theta$ does not have a singularity at $(\tau_1, \tilde{\tau})$ for every $\tau_1 \in \mathbb{T}.$ Then  $\tilde{\tau} \in \mathbb{T} \setminus E_{\Theta}$ and so by the proof of Theorem \ref{thm:generalnr}, $\mathcal{W}(S_{\Theta_{\tilde{\tau}}}) = \mathcal{W}(M_{\Theta}(\tilde{\tau}))$. Thus $\tilde{\lambda} \in \mathcal{W}(S_{\Theta_{\tilde{\tau}}}).$ This gives a contradiction since the numerical range of a compressed shift on a model space associated to a finite Blaschke product is strictly contained in $\mathbb{D}$. See pp.~$181$ of \cite{gw03} for details. Thus $\Theta$ must have a singularity at $(\tau_1, \tilde{\tau})$ for some $\tau_1 \in \mathbb{T}.$  \\

\noindent ($\Leftarrow$) Since $S_{\Theta}^1$ is a contraction, $w\big( S^1_{\Theta} \big) \le 1.$   Assume $\Theta$ has a singularity at $\tilde{\tau} = (\tilde{\tau}_1, \tilde{\tau}_2) \in \mathbb{T}^2.$ Then as $\Theta = \frac{\tilde{p}}{p}$, we must have $\tilde{p}(\tilde{\tau})=0.$  To prove the desired claim, we will show that $\tilde{\tau}_1 \in \mathcal{W}( S^1_{\Theta})$ and as $|\tilde{\tau}_1|=1$, we have  $w\big( S^1_{\Theta} \big) \ge 1.$
Write
\[ \tilde{p}(z_1, z_2) = \sum_{k=0}^{m} \tilde{p}_k(z_2) z_1^k =  \tilde{p}_m(z_2) \left( z_1^m +  \sum_{k=0}^{m-1} \frac{ \tilde{p}_k(z_2)}{ \tilde{p}_m(z_2)} z_1^k \right),\]
for one-variable polynomials $\tilde{p}_1, \dots, \tilde{p}_m.$
Note that $ \tilde{p}_{m}$ does not vanish on $\mathbb{T}.$ If it did, one could conclude that $p(0,\cdot)$ vanishes on $\mathbb{T}$, a contradiction of the fact that $p$ does not vanish on $\mathbb{D} \times \mathbb{T}.$ 
Now for each $\tau \in \mathbb{T}$, consider the one-variable polynomial 
\[ \tilde{p}(z_1, \tau) =  \tilde{p}_m(\tau) \left( z_1^m +  \sum_{k=0}^{m-1} \frac{ \tilde{p}_k(\tau)}{ \tilde{p}_m(\tau)} z_1^k \right)\]
and factor it as
\[ \tilde{p}(z_1, \tau) =  \tilde{p}_m(\tau) \prod_{k=1}^m\left( z_1- \alpha_k(\tau)\right),\]
where $\alpha_1(\tau), \dots, \alpha_m(\tau)$ are the zeros of $\tilde{p}(\cdot, \tau)$. 
Now we use the fact that the zeros of a polynomial depend continuously on its coefficients, see \cite{us77}.

Fix $\epsilon >0$. Since the coefficients $ \left \{\frac{ \tilde{p}_k(\tau)}{ \tilde{p}_m(\tau)} \right \}$ are continuous on $\mathbb{T}$, there exist $\delta_1, \delta_2 >0$ such that if  $| \tau - \tilde{\tau}_2| < \delta_1$, then 
\[ \left | \frac{ \tilde{p}_k(\tau)}{ \tilde{p}_m(\tau)} - \frac{ \tilde{p}_k(\tilde{\tau}_2)}{ \tilde{p}_m(\tilde{\tau}_2)}\right | < \delta_2 \qquad \text{ for } k=1, \dots, m-1\]
and reordering the $\alpha_k(\tau)$ if necessary
\[ |\alpha_k(\tau) - \alpha_k(\tilde{\tau}_2)| < \epsilon  \qquad \text{ for } k=1, \dots, m.\]
Without loss of generality, we can assume $\alpha_1(\tilde{\tau}_2)=\tilde{\tau}_1.$ Since $\epsilon >0$ was arbitrary and $E_{\Theta}$ is finite, the above arguments shows that 
\[ 
\begin{aligned} \tilde{\tau}_1 & \in \text{Clos} \left\{ \alpha_1(\tau): \tau \in \mathbb{T} \setminus E_{\Theta}\right\}  \\
& \subseteq  \text{Clos} \left( \bigcup_{\tau \in \mathbb{T} \setminus E_{\Theta}} \mathcal{W}(S_{\Theta_{\tau}})\right) \\
& \subseteq  \text{Clos} \left( \mathcal{W}( S^1_{\Theta})\right),
\end{aligned}
\]
where we used Equation \eqref{eqn:onevar} to show that each $\alpha_1(\tau) \in  \mathcal{W}(S_{\Theta_{\tau}})$ and Theorem \ref{thm:generalnr} to conclude the last containment. 
\end{proof}
   
\section{Example: $S^1_{\Theta}$ for Simple Rational Inner Functions} \label{sec:examples}

In this section, we illustrate Theorem \ref{thm:unitary} using a particular class of rational inner functions. Specifically, let $\Theta = \prod_{i=1}^m \theta_i,$ where each $\theta_i$ is a degree  $(1,1)$ rational inner function  with a singularity on $\mathbb{T}^2.$
In what follows, we will decompose $\mathcal{K}_{\Theta}$ into specific $M_{z_1}$- and $M_{z_2}$-invariant subspaces $\mathcal{S}_1$ and $\mathcal{S}_2$ (also called $z_1$- and $z_2$-invariant), find an orthonormal basis of $\mathcal{H}(K_2) := \mathcal{S}_2 \ominus z_2 \mathcal{S}_2$, and use this basis to compute the matrix-valued function $M_{\Theta}$ from Theorem \ref{thm:unitary}.

\subsection{Preliminaries} We first require preliminary information about degree $(1,1)$ rational inner functions with a singularity on $\mathbb{T}^2$ and their associated model spaces. To indicate that these are particularly simple functions, we denote 
them with $\theta$ rather than $\Theta.$ Then for such a $\theta,$ there is a polynomial $p(z) =a + bz_1 + cz_2 + d z_1 z_2$ with no zeros in $\mathbb{D}^2 \cup (\mathbb{T} \times \mathbb{D}) \cup (\mathbb{D} \times \mathbb{T})$ such that 
 \[ \theta(z) = \frac{\tilde{p}(z)}{p(z)} = \frac{ \bar{a}z_1 z_2 + \bar{b}z_2 + \bar{c} z_1 + \bar{d}}{ a + bz_1 + cz_2 + d z_1 z_2 }.\]
 In this situation, it is particularly easy to identify shift-invariant subspaces $\mathcal{S}_1$ and $\mathcal{S}_2$ associated to the two-variable model space $\mathcal{K}_{\theta}.$

\begin{lemma} \label{lem:AD} Let $\theta =\frac{\tilde{p}}{p}$ be a degree $(1,1)$ rational inner function with $p(z) = a + bz_1 + cz_2+d z_1 z_2$. Assume $p$ vanishes at $\tau = (\tau_1, \tau_2) \in \mathbb{T}^2.$ Then $\mathcal{K}_{\theta} = \mathcal{S}_1 \oplus \mathcal{S}_2$, where
\begin{equation} \label{eqn:subspaces} \mathcal{S}_1 = \mathcal{H}\left( \frac{g(z)\overline{g(w)}}{1-z_1\bar{w}_1} \right)  \ \ \text{ and }\ \  \mathcal{S}_2 = \mathcal{H}\left( \frac{f(z) \overline{f(w)}}{1-z_2 \bar{w}_2}\right)\end{equation}
with the functions in the reproducing kernels given by
\[ g(z) =  \frac{\gamma (z_1 - \tau_1)}{p(z)} \ \  \text{ and }  \ \   f(z) =  \frac{\lambda (z_2 - \tau_2)}{p(z)},\] 
for any $\lambda, \gamma$ satisfying $|\lambda|^2 = |\bar{a}c -d \bar{b}|$ and $|\gamma|^2 = |\bar{a}b-d\bar{c}|.$ Moreover, $\mathcal{S}_1$ and $\mathcal{S}_2$ are the only subspaces of $\mathcal{K}_{\theta}$ satisfying $\mathcal{K}_{\theta} =\mathcal{S}_1 \oplus \mathcal{S}_2$ that are respectively $z_1$- and $z_2$-invariant.
\end{lemma}

\begin{proof} Define $f$ and $g$ as above.   As mentioned earlier, by  \cite{bsv05,bk13}, there are canonical  subspaces $\mathcal{S}^{max}_1$ and $\mathcal{S}^{min}_2$ with $\mathcal{K}_{\theta}= \mathcal{S}^{max}_1 \oplus \mathcal{S}^{min}_2$ that are respectively $z_1$- and $z_2$-invariant. As they are subspaces of $H^2(\mathbb{D}^2)$, we can write them as  
\[ \mathcal{S}^{max}_1 = \mathcal{H}\left(\frac{K_1(z,w)}{1-z_1 \bar{w}_1}\right) \ \  \text{ and } \ \  \mathcal{S}^{min}_2 = \mathcal{H}\left(\frac{K_2(z,w)}{1-z_2 \bar{w}_2} \right),\] 
for Agler kernels $(K_1, K_2)$ of $\theta$ defined as in \eqref{eqn:kernel}.  Our first goal is to show that $K_1(z,w) = g(z)\overline{g(w)}$ and $K_2(z,w) = f(z)\overline{f(w)}.$ Now, by Theorems \ref{thm:dim} and  \ref{thm:rational}, there are polynomials 
\[ r(z_1) = A + Bz_1 \ \text{ and } \ q(z_2) = C + D z_2 \]
such that $K_1(z,w) = \frac{r(z_1)}{p(z)}\overline{\frac{r(w_1)}{p(w)}}$ and $K_2(z,w) = \frac{q(z_2)}{p(z)}\overline{\frac{q(w_2)}{p(w)}}$. The definition of Agler kernels implies that $(K_1,K_2)$ satisfy the formula
\begin{equation} \label{eqn:AD2} 1 - \theta(z) \overline{\theta(w)}  = (1- z_1 \overline{w_1}) \frac{q(z_2)}{p(z)} \frac{ \overline{q(w_2)}}{\overline{p(w)}}  +  (1- z_2 \overline{w_2})  \frac{r(z_1)}{p(z)} \frac{ \overline{r(w_1)}}{\overline{p(w)}}.\end{equation}
Multiplying through by $p(z) \overline{p(w)}$ and letting $(w_1, w_2) \rightarrow (\tau_1, \tau_2)$ gives
\[ 0 = \overline{q(\tau_2)} (1- z_1 \overline{\tau_1}) q(z_2) +  \overline{r(\tau_1)} (1- z_2 \overline{\tau_2}) r(z_1).   \]
This implies that $q(\tau_2)=0$ and so, $q(z_2) = F (z_2 - \tau_2)$ for some constant $F$. Similarly, $r(z_1) = G (z_1 - \tau_1)$ for some constant $G.$ To show that $K_1$ and $K_2$ have the desired expressions in terms of $g$ and $f$, we just need to show that $|F|^2 = |\lambda|^2$ and $|G|^2 = |\gamma|^2.$
 
Substituting the formulas for $q$ and $r$ into \eqref{eqn:AD2} and multiplying through by $p(z) \overline{p(w)}$ gives
\[  p(z) \overline{p(w)} - \tilde{p}(z) \overline{\tilde{p}(w)}  = (1- z_1 \overline{w_1}) |F|^2 (z_2 - \tau_2) \overline{ (w_2 - \tau_2)} +  (1- z_2 \overline{w_2}) |G|^2 (z_1 - \tau_1)\overline{ (w_1 - \tau_1)}.\]
Recalling that $p(z) = a+ bz_1+ cz_2 + dz_1z_2$ and $\tilde{p}(z) = \bar{a}z_1 z_2 + \bar{b}z_2 + \bar{c}z_1 + \bar{d},$ we can equate the coefficients of the monomials $1, z_1 \bar{w}_1, z_1$ and $z_2$ from both sides of the above equation to conclude:
\[
\begin{aligned}  
|a|^2 - |d|^2 &= |F|^2 +|G|^2\\
|b|^2 - |c|^2 &= -|F|^2 + |G|^2 \\
 \bar{a} b-d \bar{c}  &= -\overline{\tau_1} |G|^2  \\
\bar{a} c-  d \bar{b}&= -\overline{\tau_2} |F|^2.
\end{aligned}
\]
The last  two equations show $|F|^2 = |\lambda|^2$ and $|G|^2 = |\gamma|^2$, implying that $K_1(z,w) = g(z)\overline{g(w)}$ and $K_2(z,w) = f(z)\overline{f(w)}.$
In combination with the first equation, one can also obtain the useful formulas
\begin{equation} \label{eqn:zeros}  \tau_1 = \frac{-2( a\bar{b} - c \bar{d})}{ |a|^2 +|b|^2 - |c|^2 -|d|^2} \ \text{ and } \   \tau_2 = \frac{-2( a\bar{c} - b \bar{d})}{ |a|^2 +|c|^2 - |b|^2 -|d|^2}.\end{equation}
To finish the proof, observe that $p$ and  $\tilde{p}$ have two common zeros (including intersection multiplicity) on $\mathbb{T}^2$. As $\theta$ is a degree $(1,1)$  rational inner function, Corollary 13.6 in \cite{k14} implies that $\theta$ has a unique pair of Agler kernels and hence, a unique pair of decomposing subspaces $\mathcal{S}_1$ and $\mathcal{S}_2$ that are respectively $z_1$- and $z_2$-invariant. This unique pair $\mathcal{S}_1$ and $\mathcal{S}_2$ must then be the subspaces $\mathcal{S}^{max}_1$ and $\mathcal{S}^{min}_2$ found earlier.
\end{proof}

It is worth pointing out that for the function $f$ in Lemma \ref{lem:AD}, we can choose any $\lambda$ satisfying $|\lambda|^2 = |\bar{a}c -d \bar{b}|$. However, in the sequel, we will typically choose the particular $\lambda$ satisfying $\lambda^2 = \bar{a}c-d\bar{b}$. We now obtain additional information about $M_{z_1}^*$ applied to $\theta$ and this particular function $f$ from Lemma \ref{lem:AD}. 

\begin{lemma} \label{lem:bws} Let $\theta =\frac{\tilde{p}}{p}$ be a degree $(1,1)$ rational inner function with $p(z) = a + bz_1 + cz_2+d z_1 z_2$. Assume $p$ vanishes at $\tau= (\tau_1, \tau_2) \in \mathbb{T}^2$ and let $f$ be defined as in Lemma \ref{lem:AD} with $\lambda$ further satisfying $\lambda^2 = \bar{a}c-d\bar{b}$. Then
\[ 
\begin{aligned}
\left( M_{z_1}^* f \right) (z) &= f(z) \left(  -\frac{b +dz_2}{a + cz_2} \right); \\
\left( M_{z_1}^* \theta\right) (z) & = f(z) \lambda \left(  \frac{ z_2 - \tau_2}{a + cz_2}\right).
\end{aligned}
\] 
\end{lemma}
\begin{proof} First, simple computations using the definition of $f$ and $p$ give
\[ 
\begin{aligned}
(M_{z_1}^*f)(z) &=  \frac{\lambda(z_2 -\tau_2)}{z_1} \left( \frac{1}{p(z)} - \frac{1}{p(0,z_2)} \right) \\
& =  \lambda (z_2 -\tau_2)\frac{ -b-dz_2}{p(z) p(0,z_2)} \\
& = f(z) \left(  -\frac{b +dz_2}{a + cz_2} \right).
\end{aligned}
\]
Similarly, one can compute
\[ 
(M_{z_1}^*\theta)(z) = \frac{1}{z_1} \left( \frac{\tilde{p}(z)}{p(z)} - \frac{\tilde{p}(0,z_2)}{p(0,z_2)} \right).\]
Using the definitions of $p$ and $\tilde{p}$, one can obtain a common denominator, collect like terms, and cancel the $z_1$ from the denominator to obtain:
\begin{equation} \label{eqn:thetashift} (M_{z_1}^* \theta)(z) = \frac{ (\bar{a}c -d\bar{b})z_2^2 + (|a|^2 +|c|^2 - |b|^2 -|d|^2)z_2 + (a\bar{c} - b \bar{d}) }{p(z)p(0,z_2)}.\end{equation}
Recall that $\tau_2 \in \mathbb{T}$. Then using the formula for $\tau_2$ from \eqref{eqn:zeros}, one can conclude that $a\bar{c} - b \bar{d} \ne 0$ and 
\begin{equation}\label{eqn:tau} \tau_2 = \frac{1}{\overline{\tau_2}} = - \frac{|a|^2 +|c|^2 - |b|^2 -|d|^2}{2(\bar{a} c-\bar{b} d)}  \ \ \text{ and } \ \ \tau_2^2 = \frac{\tau_2}{\overline{\tau_2}} = \frac{a \bar{c} -b\bar{d}}{ \bar{a}c - d \bar{b}}.\end{equation}
Taking the numerator from \eqref{eqn:thetashift} and factoring out $(\bar{a}c -d\bar{b})$ gives

\[
\begin{aligned}
 (\bar{a}c -d\bar{b})z_2^2 &+ (|a|^2 +|c|^2 - |b|^2 -|d|^2)z_2 + (a\bar{c} - b \bar{d}) \\
 &= (\bar{a}c -d\bar{b}) \left( z_2^2 + \frac{|a|^2 +|c|^2 - |b|^2 -|d|^2}{\bar{a}c -d\bar{b}} z_2+ \frac{ a\bar{c} - b \bar{d}}{\bar{a}c -d\bar{b}}\right) \\ 
 & = \lambda^2 \left( z_2^2 - 2 \tau_2 z_2 + \tau_2^2\right).
\end{aligned}
\]
Combining our formulas gives
\[ (M_{z_1}^* \theta)(z)  = \lambda^2 \frac{ (z_2 - \tau_2)^2}{p(z) p(0,z_2)} = \lambda f(z) \left( \frac{z_2 -\tau_2}{a +c z_2}\right),\]
the desired equality.\end{proof}

\subsection{$M_\Theta$ for product $\Theta$} 

Let us now return to the question posed at the beginning of the section. Let $\Theta = \prod_{i=1}^m \theta_i,$ where each $\theta_i$ is a degree  $(1,1)$ rational inner function  with a singularity on $\mathbb{T}^2.$ We can now use Lemma \ref{lem:AD} to  decompose $\mathcal{K}_{\Theta}$ into specific $z_1$- and $z_2$- invariant subspaces $\mathcal{S}_1$ and $\mathcal{S}_2$ and find an orthonormal basis of $\mathcal{H}(K_2):= \mathcal{S}_2 \ominus z_2 \mathcal{S}_2$. Then using Lemma \ref{lem:bws}, we will compute the matrix function $M_{\Theta}$ from Theorem \ref{thm:unitary}.

For each $i$, let $\mathcal{S}_{1,\theta_i}$, $\mathcal{S}_{2, \theta_i}$, and $f_i$ denote the canonical subspaces and reproducing function associated to $\theta_i$ in Lemma \ref{lem:AD}. Then:

\begin{proposition} \label{prop:productform}
Let $\Theta =\prod_{i=1}^m \theta_i,$ where each $\theta_i$ is a degree  $(1,1)$ rational inner function $\frac{\tilde{p}_i}{p_i}$ where $p_i(z) = a_i + b_i z_1 + c_i z_2 + d_i z_1 z_2$ with a singularity at $(\tau_{1,i}, \tau_{2,i}) \in \mathbb{T}^2.$ Define
\begin{equation} \label{eqn:decompK}
\mathcal{S}_1 := \bigoplus_{i=1}^m\Big( \big({\textstyle \prod_{k=1}^{i-1}\theta_k} \big)  \mathcal{S}_{1, \theta_i} \Big) \ \text{ and } \
\mathcal{S}_2 := \bigoplus_{i=1}^m\Big( \big({\textstyle \prod_{k=1}^{i-1}\theta_k} \big) \mathcal{S}_{2, \theta_i} \Big).
\end{equation}
Then $\mathcal{K}_{\Theta} = \mathcal{S}_1 \oplus \mathcal{S}_2$ and $\mathcal{S}_1$, $\mathcal{S}_2$ are respectively $z_1$- and $z_2$-invariant. Furthermore, if $
\mathcal{H}(K_2) = \mathcal{S}_2 \ominus z_2 \mathcal{S}_2$, then the set
\begin{equation} \label{eqn:onbasis} \left \{ f_1, \ \theta_1 f_2, \ \theta_1 \theta_2 f_3, \dots, \big({\textstyle \prod_{k=1}^{m-1}\theta_k} \big) f_m \right \}\end{equation}
is an orthonormal basis for $\mathcal{H}(K_2)$, where each $f_i(z) =  \frac{\lambda_i (z_2 - \tau_{2,i})}{p_i(z)}$ and $\lambda^2_i = \overline{a_i}c_i-d_i\overline{b_i}$. 
\end{proposition}

\begin{proof} Observe that 
\begin{equation} \label{eqn:decomposition2} \mathcal{K}_{\Theta} = \mathcal{K}_{\theta_1} \oplus \theta_1 \mathcal{K}_{\theta_2} \oplus (\theta_1 \theta_2) \mathcal{K}_{\theta_3} \oplus \dots \oplus  \big({\textstyle \prod_{k=1}^{m-1}\theta_k} \big) \mathcal{K}_{\theta_m}  = \bigoplus_{i=1}^m   \big({\textstyle \prod_{k=1}^{i-1}\theta_k} \big)\mathcal{K}_{\theta_i}.\end{equation}
This can be seen by observing that the subspaces in \eqref{eqn:decomposition2} are orthogonal to each other and their reproducing kernels add to that of $\mathcal{K}_{\Theta}.$
Now by Lemma \ref{lem:AD}, we can write each $ \mathcal{K}_{\theta_i} =  \mathcal{S}_{1, \theta_i} \oplus \mathcal{S}_{2, \theta_i},$ 
where these subspaces are respectively $z_1$- and $z_2$-invariant. 
Define $\mathcal{S}_1$ and $\mathcal{S}_2$ as in \eqref{eqn:decompK}.
Then, $\mathcal{S}_1$ is an orthogonal sum of $z_1$-invariant subspaces and so is also a $z_1$-invariant subspace. Similarly, $\mathcal{S}_2$ is $z_2$-invariant. By \eqref{eqn:decomposition2}, it immediately follows that $\mathcal{K}_{\Theta} =  \mathcal{S}_1 \oplus \mathcal{S}_2.$ To prove the orthonormal basis result, observe that the components of $\mathcal{S}_2$ in \eqref{eqn:decompK} are pairwise-orthogonal and each is $z_2$-invariant. Thus
\[ \begin{aligned}
\mathcal{S}_2 \ominus z_2 \mathcal{S}_2 & =   \bigoplus_{i=1}^m \Big(  \big({\textstyle \prod_{k=1}^{i-1}\theta_k} \big) \mathcal{S}_{2,\theta_i} \ominus z_2( \big({\textstyle \prod_{k=1}^{i-1}\theta_k} \big)\mathcal{S}_{2,\theta_i}\Big)  \\
& =  \bigoplus_{i=1}^m  \big({\textstyle \prod_{k=1}^{i-1}\theta_k} \big)\Big(  \mathcal{S}_{2,\theta_i} \ominus z_2\mathcal{S}_{2,\theta_i}\Big), 
\end{aligned}
\]
where we used the fact that each ${\textstyle \prod_{k=1}^{i-1}\theta_k}$ is inner.
By the reproducing kernel formula in Lemma \ref{lem:AD}, each singleton set $\{f_i\}$ is an orthonormal basis for $\mathcal{S}_{2,\theta_i} \ominus z_2 \mathcal{S}_{2,\theta_i}$. Thus 
each singleton set $ \left \{  \big({\textstyle \prod_{k=1}^{i-1}\theta_k} \big) f_i \right \} $ is an orthonormal basis for  $ \big({\textstyle \prod_{k=1}^{i-1}\theta_k} \big)\left(\mathcal{S}_{2,\theta_i} \ominus z_2 \mathcal{S}_{2,\theta_i}\right) $. Since the decomposition of $\mathcal{S}_2 \ominus z_2 \mathcal{S}_2 $ into components in the above equation is orthogonal, the set
$ \left \{ f_1, \theta_1 f_2, \dots,  \big({\textstyle \prod_{k=1}^{m-1}\theta_k} \big) f_m \right \}$ gives the desired orthonormal basis.
\end{proof}

Recall that $\deg \Theta = (m,m)$.  By Theorem \ref{thm:unitary},  the operator $S_{\Theta}^1 := P_{\mathcal{S}_2} M_{z_1} |_{\mathcal{S}_2}$ is unitarily equivalent to a $z_2$ matrix-valued Toeplitz operator with $m\times m$ symbol $M_{\Theta}$, whose entries are rational in $\bar{z}_2$ and continuous on $\overline{\mathbb{D}}.$  For this particular $\Theta$ and $\mathcal{S}_2$, we can compute $M_{\Theta}$:

\begin{theorem} \label{thm:nr1} Let $\Theta = \prod_{i=1}^m \theta_i,$ where each $\theta_i$ is a degree  $(1,1)$ rational inner function $\frac{\tilde{p}_i}{p_i}$ where $p_i(z) = a_i + b_i z_1 + c_i z_2 + d_i z_1 z_2$ has a zero at $\tau_i = (\tau_{1,i}, \tau_{2,i}) \in \mathbb{T}^2$. Let $\mathcal{S}_2$ be as in \eqref{eqn:decompK}. Then, the $m \times m$ matrix-valued function $M_{\Theta}$ from Theorem \ref{thm:unitary} is given entry-wise by
\[ M_{\Theta} (z_2)_{ji} =  \left \{ \begin{array}{cc}
\overline{ \lambda_j \left(  \frac{ z_2 - \tau_{2,j}}{a_j + c_jz_2}\right)  \lambda_i \left(  \frac{ z_2 - \tau_{2,i}}{a_i + c_iz_2}\right) \prod_{k=i+1}^{j-1} \left(\frac{\overline{b_k}z_2 + \overline{d_k}}{a_k + c_k z_2}\right)} & \text{ if $j>i$}; \\
& \\
\overline{\left(  -\frac{b_i +d_iz_2}{a_i + c_iz_2} \right) } & \text{ if $j=i$}; \\
&\\ 
0 & \text{ if } j<i,
\end{array}
\right. \]
where each $\lambda_i$ satisfies $\lambda_i^2 = \overline{a_i}c_i-d_i\overline{b_i}$.
\end{theorem} 

\begin{proof}  By the proof of Theorem \ref{thm:unitary}, we need only show that this $M_{\Theta}$ satisfies the correct formula.  Specifically, let \[ \left \{ f_1, \  \theta_1f_2, \ \theta_1\theta_2 f_3, \dots,  \big({\textstyle \prod_{k=1}^{m-1}\theta_k} \big)  f_m \right\} \] 
denote the orthonormal basis of $\mathcal{H}(K_2)= \mathcal{S}_2 \ominus z_2 \mathcal{S}_2$ from Proposition \ref{prop:productform}.  Then by the proof of Theorem \ref{thm:unitary}, $M_{\Theta} = H^*$ where $H$ is the $m \times m$ matrix of unique functions $h_{ij} \in H^2_2(\mathbb{D})$ satisfying
\[ M_{z_1}^*\left(  \big({\textstyle \prod_{k=1}^{j-1}\theta_k} \big)  f_j \right) = f_1 h_{1j}  + \dots + \big({\textstyle \prod_{k=1}^{m-1}\theta_k} \big) f_m h_{mj}, \quad \text{ for } j=1, \dots, m.\]  
Then to identify each $h_{ij}$ we need only write 
\[  M_{z_1}^*\left(  \big({\textstyle \prod_{k=1}^{j-1}\theta_k} \big)  f_j \right)  = \big({\textstyle \prod_{k=1}^{i-1}\theta_k} \big)  f_i  h +g,\]
where $h \in H_2^2(\mathbb{D})$ and $g \perp \big({\textstyle \prod_{k=1}^{i-1}\theta_k} \big) \mathcal{K}_{\theta_i}.$ Then, we would have $h_{ij} = h.$ 

To begin computing the $h_{ij}$, fix $i,j$ with $i >j$. Observe that the following subspace of $\mathcal{S}_2$
\[   \bigoplus_{\ell=1}^j   \big({\textstyle \prod_{k=1}^{\ell-1}\theta_k} \big)\mathcal{K}_{\theta_{\ell}}\]
is the two-variable model space associated to the inner function ${\textstyle \prod_{k=1}^{j-1}\theta_k}$ and hence, is invariant under $M_{z_1}^*$. Thus if $i >j$, the fact that
\[M_{z_1}^*\left(  \big({\textstyle \prod_{k=1}^{j-1}\theta_k} \big)  f_j \right) \in \bigoplus_{\ell=1}^j   \big({\textstyle \prod_{k=1}^{\ell-1}\theta_k} \big)\mathcal{K}_{\theta_\ell} \perp   \big({\textstyle \prod_{k=1}^{i-1}\theta_k} \big)\mathcal{K}_{\theta_i}\]
 implies that  $h_{ij}\equiv 0.$ 
For the other cases, we will use the identity
\begin{equation} \label{eqn:bwsform} M_{z_1}^* \left( GH \right) = H M_{z_1}^*( G) +G(0, z_2)M_{z_1}^*(H),\end{equation}
for any $G,H \in H^2(\mathbb{D}^2)$ with $GH \in H^2(\mathbb{D}^2)$.  Now fix $i=j$ and observe that by \eqref{eqn:bwsform} and Lemma \ref{lem:bws},
\[ 
\begin{aligned}
M_{z_1}^* \left ( \big({\textstyle \prod_{k=1}^{j-1}\theta_k} \big) f_j \right) &=\big({\textstyle \prod_{k=1}^{j-1}\theta_k} \big) M_{z_1}^*(f_j) + f_j(0,z_2) M_{z_1}^*\big({\textstyle \prod_{k=1}^{j-1}\theta_k} \big)\\
& =\big({\textstyle \prod_{k=1}^{j-1}\theta_k} \big) f_j \left(  -\frac{b_j +d_jz_2}{a_j + c_jz_2} \right)  +  f_j(0,z_2) M_{z_1}^* \big({\textstyle \prod_{k=1}^{j-1}\theta_k} \big).
\end{aligned}
\]
The second term is in the model space associated to ${\textstyle \prod_{k=1}^{j-1}\theta_k}$ and hence, is orthogonal to  $\big({\textstyle \prod_{k=1}^{j-1}\theta_k} \big)\mathcal{K}_{\theta_j}$. This follows from Proposition $3.5$ in \cite{b13}, which show that if $\phi$ is an inner function on $\mathbb{D}^2$, then $\left( M_{z_1}^*\phi \right)H^2_2(\mathbb{D}) \subseteq \mathcal{K}_{\phi}$. Thus, we can conclude that 
\[ h_{jj}(z_2) =    -\frac{b_j +d_jz_2}{a_j + c_jz_2}.\]
Lastly, fix $i,j$ with $i <j.$ Then by applying \eqref{eqn:bwsform} again, we have 
\[ 
\begin{aligned}
M_{z_1}^*\left(\big({\textstyle \prod_{k=1}^{j-1}\theta_k} \big)  f_j \right) & = \big({\textstyle \prod_{k=1}^{i}\theta_k} \big) M_{z_1}^*\left(  \big({\textstyle \prod_{k=i+1}^{j-1}\theta_k} \big)   f_j \right) +  \left( \big({\textstyle \prod_{k=i+1}^{j-1}\theta_k} \big)f_j \right)(0,z_2) M_{z_1}^*\big({\textstyle \prod_{k=1}^{i}\theta_k} \big) \\
&= \big({\textstyle \prod_{k=1}^{i}\theta_k} \big) M_{z_1}^*\left ( \big({\textstyle \prod_{k=i+1}^{j-1}\theta_k} \big) f_j\right) + \left( \big({\textstyle \prod_{k=i+1}^{j-1}\theta_k} \big) f_j \right)(0,z_2) \big({\textstyle \prod_{k=1}^{i-1}\theta_k} \big) M_{z_1}^*(\theta_{i}) \\
& \qquad+  \left( \big({\textstyle \prod_{k=i}^{j-1}\theta_k} \big) f_j\right)(0,z_2)M_{z_1}^*\big({\textstyle \prod_{k=1}^{i-1}\theta_k} \big).
\end{aligned}
\]
Let us consider the terms in the last sum. The first term lies in $\big({\textstyle \prod_{k=1}^{i-1}\theta_k} \big)\theta_i H^2(\mathbb{D}^2)$ and so is orthogonal to $\big({\textstyle \prod_{k=1}^{i-1}\theta_k} \big) \mathcal{K}_{\theta_i}.$ 
Similarly, the third term is in the model space associated to $\big({\textstyle \prod_{k=1}^{i-1}\theta_k} \big)$ by Proposition $3.5$ in \cite{b13} and so is orthogonal to $\big({\textstyle \prod_{k=1}^{i-1}\theta_k} \big) \mathcal{K}_{\theta_i}.$ 
Thus, the second term is the only one that contributes to $h_{ij}$. By Lemma \ref{lem:bws}, we can replace $M_{z_1}^* \theta_i$ in the second term to obtain
\[  \left( \big({\textstyle \prod_{k=i+1}^{j-1}\theta_k}\big) f_j\right)(0,z_2)\lambda_i \left(  \frac{ z_2 - \tau_{2,i}}{a_i + c_iz_2}\right) \big({\textstyle \prod_{k=1}^{i-1}\theta_k} \big) f_i. \]
It follows that
\[ 
\begin{aligned}
h_{ij}(z_2) &=  \left( \big({\textstyle \prod_{k=i+1}^{j-1}\theta_k} \big) f_j \right)(0,z_2)\lambda_i \left(  \frac{ z_2 - \tau_{2,i}}{a_i + c_iz_2}\right) \\
& = \lambda_j \left(  \frac{ z_2 - \tau_{2,j}}{a_j + c_jz_2}\right)  \lambda_i \left(  \frac{ z_2 - \tau_{2,i}}{a_i + c_iz_2}\right) {\prod_{k=i+1}^{j-1} }\left(\frac{\overline{b_k}z_2 + \overline{d_k}}{a_k + c_k z_2}\right),
\end{aligned}
\]
where we used the formulas for each $\theta_k$ and $f_j$. 
Thus, $H$ is defined entry-wise by 
\[ H(z_2)_{ij} =  \left \{ \begin{array}{cc}
\lambda_j \left(  \frac{ z_2 - \tau_{2,j}}{a_j + c_jz_2}\right)  \lambda_i \left(  \frac{ z_2 - \tau_{2,i}}{a_i + c_iz_2}\right) \prod_{k=i+1}^{j-1} \left(\frac{\overline{b_k}z_2 + \overline{d_k}}{a_k + c_k z_2}\right) & \text{ if $i<j$}; \\
& \\
\left(  -\frac{b_i +d_iz_2}{a_i + c_iz_2} \right)  & \text{ if $i=j$}; \\
&\\ 
0 & \text{ if } i>j.
\end{array}
\right. \]
Then the fact that $M_{\Theta} = H^*$ gives the desired formula. \end{proof}

To make this concrete, we  compute several $M_{\Theta}$ using the formula from Theorem \ref{thm:nr1}.

\begin{example} First, let $\Theta$ be the following degree $(2,2)$ rational inner function:
\[ 
\Theta(z) = \theta_1(z) \theta_2(z) = \left( \frac{2z_1z_2-z_1-z_2}{2-z_1-z_2} \right) \left( \frac{3z_1z_2-2z_1-z_2}{3-z_1-2z_2} \right).\]
Then $\tau_{2,1} = \tau_{2,2}=1$ and we can take $\lambda_1 = i \sqrt{2}$ and $\lambda_2 = i \sqrt{6}$. Then, by Theorem \ref{thm:nr1}:
\[ M_{\Theta}(z_2) =
\left[   { \setstretch{3}
 \begin{array}{cc}
\dfrac{1}{2-\bar{z}_2} & 0 \\
\dfrac{ -\sqrt{12}(1-\bar{z}_2)^2}{(2-\bar{z}_2)(3-2\bar{z}_2)} & \dfrac{1}{3-2\bar{z}_2} 
\end{array} }
\right].\]
Thus, $S^1_{\Theta}$ is unitarily equivalent to the matrix-valued Toeplitz operator with this symbol.
\end{example}

\begin{example} Now, let $\Theta$ be the following degree $(3,3)$ rational inner function:
\[\Theta(z) = \theta_1(z) \theta_2(z)\theta_3(z) = \left( \frac{2z_1z_2-z_1-z_2}{2-z_1-z_2} \right) \left( \frac{3z_1z_2-2z_1-z_2}{3-z_1-2z_2} \right) \left(\frac{3z_1z_2-z_1-z_2-1}{3-z_1-z_2-z_1z_2} \right).\]
Then $\tau_{2,1} = \tau_{2,2}=\tau_{2,3}=1$ and we have $\lambda_1 = i \sqrt{2}$, $\lambda_2 = i \sqrt{6}$, and $\lambda_3 = 2i.$ By Theorem \ref{thm:nr1}:
\[ M_{\Theta}(z_2) =
\left[   { \setstretch{3}
 \begin{array}{ccc}
\dfrac{1}{2-\bar{z}_2} & 0 & 0  \\
\dfrac{ -\sqrt{12}(1-\bar{z}_2)^2}{(2-\bar{z}_2)(3-2\bar{z}_2)} & \dfrac{1}{3-2\bar{z}_2} & 0\\
\dfrac{2 \sqrt{2} (1-\bar{z}_2)^2(1+\bar{z}_2)}{(2-\bar{z}_2)(3-\bar{z}_2)(3-2\bar{z}_2)} & \dfrac{-2\sqrt{6} (1-\bar{z}_2)^2}{(3-2\bar{z}_2)(3-\bar{z}_2)}  & \dfrac{1+\bar{z}_2}{3-\bar{z}_2} 
\end{array} }
\right],\]
so  $S^1_{\Theta}$ is unitarily equivalent to the matrix-valued Toeplitz operator with this symbol.

It is worth pointing that out that these $M_{\Theta}$ are lower triangular (rather than upper triangular like \eqref{eqn:onevar}) because in our computations, we ordered our bases in a different way than is typically done in the one-variable situation.
\end{example}

\section{Zero Inclusion Question for the Numerical Range}
\label{Zero}

In this section, we study the question of when zero is in the numerical range associated to a product of two degree $(1,1)$ rational inner functions: Using $\mathcal{S}_2$ as defined in \eqref{eqn:decompK} and recalling that $ S^1_{\Theta} : = P_{\mathcal{S}_2}  M_{z_1}|_{\mathcal{S}_2}$, we are interested in the question of when zero is in $\mathcal{W}\big( S^1_{\Theta}\big)$. 

We begin with some notation. Let $\Theta = \theta_1 \theta_2$, where each $\theta_j$ is a degree  $(1,1)$ rational inner function $\frac{\tilde{p}_j}{p_j}$ and each $p_j(z) = a_j + b_j z_1 + c_j z_2 + d_j z_1 z_2$ has a zero at $\tau_j = (\tau_{1,j}, \tau_{2,j}) \in \mathbb{T}^2$. By Corollary \ref{thm:nr}, 
\[ \text{Clos}\big(\mathcal{W}( S^1_{\Theta})\big) = \text{Conv}\Big( \bigcup_{\tau \in \mathbb{T}} \mathcal{W}(M_{\Theta}(\tau))\Big),\] 
where $M_{\Theta}$ is the $2 \times 2$ matrix-valued function given in Theorem~\ref{thm:nr1}. For our $\Theta=\theta_1\theta_2$, 
\begin{equation}\label{eqn:2by2}
M_\Theta(z_2) =
  \begin{bmatrix}
    -\overline{\left(\frac{b_1 + d_1 z_2}{a_1 + c_1 z_2}\right)} & 0\\
    &\\
    \bar{\lambda}_1 \bar{\lambda}_2 
    \overline{\left({\frac{z_2 - \tau_{2, 1}}{a_1 + c_1 z_2}}\right)} \cdot \overline{\left({\frac{z_2 - \tau_{2,2}}{a_2 + c_2 z_2}}\right)} & -\overline{\left(\frac{b_2 + d_2 z_2}{a_2 + c_2 z_2}\right)}
  \end{bmatrix},
\end{equation}
where $\lambda_j^2= \bar{a}_j c_j - d_j \bar{b}_j$, for $j = 1, 2$. In future computations, we let $f_{j, z_2}$ denote $(j,j)$-entry of $M_\Theta(z_2)$ and let $\beta_j$ denote the center of the circle $\mathcal{C}_j := \{ -\frac{b_j + d_j z}{a_j + c_j z}: z \in \mathbb{T}\}$ for  $j=1,2.$ By Theorem~\ref{thm:numericalradius}, the numerical radius $w( S^1_{\Theta} ) =1$ and so, the entries (eigenvalues) $f_{1, z_2}$ and $f_{2, z_2}$ as well as the entire circles $\mathcal{C}_1$  and $\mathcal{C}_2$ are in $\overline{\mathbb{D}}$.  
For each $z_2 \in \mathbb{T}$, define
\[U_{\alpha_{z_2}} :=
  \begin{bmatrix}
   0 & e^{i \alpha_{z_2} }
\\
1 & 0
  \end{bmatrix},
\]
where $\alpha_{z_2} \in \mathbb{R}$ is chosen so that 
\begin{equation} \label{eqn:pos12} \widehat{M}_{\Theta}(z_2): = U_{\alpha_{z_2}}^* M_\Theta(z_2) U_{\alpha_{z_2}} \end{equation}
has positive $(1, 2)$-entry. Since $U_{\alpha_{z_2}}$ is unitary, $\widehat{M}_{\Theta}(z_2)$ and $M_\Theta(z_2)$ have the same numerical range. We will often apply the Elliptical Range Theorem (see, for example, \cite{ckli}), which says that the numerical range of a $2 \times 2$ upper-triangular matrix
\begin{equation}\label{ERT} A =
  \begin{bmatrix}
   a &  c 
\\
   0& b
  \end{bmatrix}
\end{equation}
is an elliptical disk with foci at $a$ and $b$ and minor axis of length $|c| = \left({\mbox{trace} (A^* A) - |a|^2 - |b|^2}\right)^{1/2}$. 
In particular, the numerical range of $ \widehat{M}_{\Theta}(z_2)$, and hence of $M_{\Theta}(z_2)$, is an elliptical disk with foci
$f_{1,z_2}$ and $f_{2,z_2}$ and minor axis length:
 \begin{equation} \label{eqn:minor_axis} m_{z_2} := |\lambda_1\,  \lambda_2| \left|\frac{z_2 - \tau_{2, 1}}{a_1 + c_1 z_2}\right| \cdot \left|\frac{z_2 - \tau_{2, 2}}{a_2 + c_2z_2}\right|.\end{equation}

\subsection{When is $0$ in the numerical range?} Now let us consider the zero inclusion question.

\begin{proposition}\label{lem:general} Let $\Theta = \theta_1 \theta_2,$ where each $\theta_j$ is a degree  $(1,1)$ rational inner function $\frac{\tilde{p}_j}{p_j}$ where $p_j(z) = a_j + b_j z_1 + c_j z_2 + d_j z_1 z_2$ has a zero at $\tau_j = (\tau_{1,j}, \tau_{2,j}) \in \mathbb{T}^2$. If there exists $\gamma \in \mathbb{T} \setminus \{\tau_{2, j}\}_{j = 1, 2}$ such that either
 \begin{equation}\label{item:equal}  
|f_{1, \gamma}| + |f_{2, \gamma}| < |1 - \bar{f}_{1, \gamma} f_{2, \gamma}|
\end{equation}
or  
\begin{equation}\label{item:notequal} 
|\bar{\beta}_j - f_{j, \gamma}|  > |\bar{\beta}_j| \quad \text{for  $j=1$ or $j=2$,} \end{equation} 
then $0 \in \mathcal{W}( S^1_{\Theta})^0$. Furthermore, if $a_j \bar{b}_j - c_j \bar{d}_j \in \mathbb{R}$, then  \eqref{item:notequal} holds if and only if
\[|\bar{a}_j d_j - b_j \bar{c}_j| > |a_j \bar{b}_j - c_j \bar{d}_j|.\] 
 \end{proposition}
 
 \begin{proof}  We first perform a general computation for any $z_2 \in  \mathbb{T}$. By the discussions preceding \eqref{eqn:minor_axis}, the numerical range of $M_{\Theta}(z_2)$ is an elliptical disk with foci $f_{1,z_2}$ and $f_{2,z_2}$ and minor axis $m_{z_2}$ given in \eqref{eqn:minor_axis}.
 We will show that  $m_{z_2}$ also satisfies
 \begin{equation} \label{eqn:minor_axis2} m_{z_2} = \sqrt{1 - |f_{1, z_2}|^2}\sqrt{1 - |f_{2, z_2}|^2}.
 \end{equation}
To this end, observe that
\[
{1 - |f_{j, z_2}|^2}  =  \frac{|a_j|^2 + |c_j|^2 - |b_j|^2 - |d_j|^2  + (\overline{a}_j c_j - \overline{b}_j d_j) z_2 + (a_j \overline{c}_j - b_j \overline{d}_j) \overline{z}_2 }{|a_j + c_j z_2|^2}.\] 
From \eqref{eqn:tau}, we know that each $\tau_{2,j} = \frac{1}{\overline{\tau_{2,j}}} = - \frac{|a_j|^2 +|c_j|^2 - |b_j|^2 -|d_j|^2}{2(\bar{a}_j c_j-\bar{b}_j d_j)}$. This implies that
\[{1 - |f_{j, z_2}|^2}  = (|a_j|^2 + |c_j|^2 - |b_j|^2 - |d_j|^2) \frac{1 -  \frac{\bar{\tau}_{2, j}}{2} z_2 - \frac{{\tau}_{2, j}}{2} \bar{z}_2}{|a_j + c_j z_2|^2}.
\]
Since both $\tau_{2, j}\in \mathbb{T}$ and $z_2 \in \mathbb{T}$, and as $|f_{j,z_2}| \le 1$, we have
\[ \begin{aligned}
1 - |f_{j, z_2}|^2 & =  \frac{(\bar{a}_j c_j-\bar{b}_j d_j) \bar{z}_2 (z_2 - \tau_{2, j})^2}{|a_j + c_j z_2|^2} 
& = \frac{|\bar{a}_j c_j-\bar{b}_j d_j| \, |z - \tau_{2, j}|^2}{{|a_j + c_j z_2|^2}} =
\frac{|\lambda_j|^2 \, |z - \tau_{2, j}|^2}{{|a_j + c_j z_2|^2}},
\end{aligned}
\]
where $\lambda_j$ is defined as in Theorem~\ref{thm:nr1}. This proves \eqref{eqn:minor_axis2}. Then a simple computation using the definition of an ellipse
shows that the ellipse bounding $M_{\Theta}(z_2)$ 
has major axis given by
\begin{equation} \label{eqn:major} M_{z_2} := |1 - \overline{f}_{1, z_2} f_{2, z_2}|.\end{equation}

Now, to establish the first claim, assume there is some $\gamma  \in \mathbb{T} \setminus \{\tau_{2, j}\}_{j = 1, 2}$ satisfying \eqref{item:equal}. One can see that the ellipse bounding $\mathcal{W}( M_{\Theta}(z_2))$ is non-degenerate because \eqref{eqn:minor_axis} implies $m_{\gamma} \ne 0.$  Then combining condition~\eqref{item:equal} with the formula for the major axis \eqref{eqn:major} immediately gives $0 \in \mathcal{W}(M_\Theta(\gamma))^0 \subseteq \mathcal{W}( S^1_{\Theta})^0$.
  
To establish the second claim, assume there is some $\gamma  \in \mathbb{T} \setminus \{\tau_{2, j}\}_{j = 1, 2}$ such that a focus $f_{j,\gamma}$ satisfies  \eqref{item:notequal}.  Then, since $\bar{\beta}_j$ is the center of the circle on which the $f_{j,z_2}$ lie, $|\bar{\beta}_j - f_{j, z_2}| > |\bar{\beta}_j|$ for all $z_2 \in \mathbb{T}$.   Thus, the convex hull of the foci contains zero and since each $|\bar{\beta_j} - f_{j,z_2}| > 0$, we know that zero lies in $\mathcal{W}( S^1_{\Theta})^0$.

Now suppose that $a_j \bar{b}_j - c_j \bar{d}_j \in \mathbb{R}$ and consider the circle $\{ -\frac{b_j + d_j z}{a_j + c_j z}: z \in \mathbb{T}\}$ with center $\beta_j$.  If $z \in \mathbb{T}$ and $w = -\frac{b_j + d_j z}{a_j + c_j z}$, then a computation gives
 \[z = - \frac{b_j + a_j w}{d_j + c_j w}.\]
 Since $|z|^2 = 1$, we have
 \[1 =  - \frac{b_j + a_j w}{d_j + c_j w} \cdot \overline{- \frac{b_j + a_j w}{d_j + c_j w}},\] 
 which implies
   \[(|a_j|^2 - |c_j|^2)|w|^2 + 2(a_j \bar{b}_j - c_j \bar{d}_j) \Re w  + (|b_j|^2 - |d_j|^2) = 0.\]
 Writing $w = x + i y$, completing the square and computing, we see that the center $\beta_j$ satisfies
 \[\beta_j = -\frac{a_j \bar{b}_j - c_j \bar{d}_j}{|a_j|^2 - |c_j|^2},\] so $\beta_j = \bar{\beta}_j$ by our assumption that $a_j \bar{b}_j - c_j \bar{d}_j$ is real. The radius of  the circle is
 \[\left|\frac{|d_j|^2 - |b_j|^2}{|a_j|^2 - |c_j|^2} + \frac{(a_j \bar{b}_j - c_j \bar{d}_j)^2}{(|a_j|^2 - |c_j|^2)^2}\right|^{1/2} = \frac{|a_j \bar{d}_j - \bar{b}_j c_j|}{|a_j|^2 - |c_j|^2},\] where we used the fact that our assumptions imply $|a_j| \ne |c_j|$. Thus, condition~\eqref{item:notequal} holds if and only if
 
 \[|a_j \bar{d}_j - \bar{b}_j c_j| > |a_j \bar{b}_j - c_j \bar{d}_j|,\]
 as desired.
 \end{proof}
 
\begin{remark} In the first part of Proposition~\ref{lem:general}, when zero lies in the interior of a single ellipse, we can say more if the foci $f_{1, \gamma}$ and $f_{2, \gamma}$ lie on a line through the origin.  First, if the line segment joining the foci contains the origin in its interior, then condition~\eqref{item:equal} implies that the ellipse is nondegenerate and zero immediately lies in the interior of the ellipse. A similar argument can be made if one or both of the foci is zero.

If the foci lie on a line that passes through the origin and are in the same quadrant, we can write $f_{1, \gamma} = r_1 e^{i \phi}$ and $f_{2, \gamma} =   r_2 e^{i \phi}$ with $r_1, r_2 > 0$. Since the numerical radius is at most $1$, we know that $r_j \le 1$ for $j = 1, 2$. If either $r_1 = 1$ or $r_2 = 1$, then condition \eqref{item:equal} cannot hold.  Thus $0 \le r_j < 1$ for $j = 1, 2$ and condition \eqref{item:equal} holds if and only if
 $r_1 + r_2 < 1 - r_1 r_2$, which happens precisely when 
 $r_1 < \frac{1 - r_2}{1 + r_2}.$
 \end{remark}

Before proceeding further, we require the following lemma:

\begin{proposition}\label{prop:elementary} Let $\Theta = \frac{\tilde{p}}{p}$ be a rational inner function on $\mathbb{D}^2$,  where $p(z) = a + bz_1 + cz_2$ is a polynomial with a zero on $\mathbb{T}^2$. Then $|a| = |b| + |c|$. \end{proposition}

\begin{proof} 
Since $\Theta$ is holomorphic, the polynomial $p$ does not vanish inside $\mathbb{D}^2$. If $|a| < |b| + |c|$, we could choose $z_1$ and $z_2$ to make $p$ vanish in $\mathbb{D}^2$, so this is impossible. Thus, we know that $|a| \ge |b| + |c|$. But $p$ has a zero $(\tau_1, \tau_2)$ on $\mathbb{T}$. Thus, $a = -b \tau_1 - c \tau_2$ and so $|a| \le |b| + |c|$. Combining these two inequalities, we obtain $|a| = |b| + |c|$.
\end{proof}

In the following proposition, we restrict to the situation where  $\Theta = \theta_1 \theta_2$ and each $\theta_j=\frac{\tilde{p}_j}{p_j}$  with $p_j(z) = a_j + b_j z_1 + c_j z_2$. We further require that $b_j <0$ and $a_j, c_j >0$.  Given these assumptions, one can divide through by $|b_j|$ and automatically assume $b_j =-1.$ 

We can now answer the zero inclusion question using the coefficients of the polynomials defining $\Theta$ as follows:
 
\begin{proposition}\label{lem:positive_eigenvalues} Let $\Theta = \theta_1 \theta_2,$ where each $\theta_j$ is a degree  $(1,1)$ rational inner function $\frac{\tilde{p}_j}{p_j}$ where $p_j(z) = a_j - z_1 + c_j z_2$ has a zero at $\tau_j = (\tau_{1,j}, \tau_{2,j}) \in \mathbb{T}^2$ and $a_j, c_j >0.$ Then

\begin{enumerate}
\item[$\bullet$] $0  \in \mathcal{W}( S^1_{\Theta})^0$ if and only if $c_1 c_2 > \frac{1}{2}$;
\item[$\bullet$] $0 \in \partial \mathcal{W}( S^1_{\Theta})$ if and only if $c_1 c_2 = \frac{1}{2}$.
\end{enumerate}
\end{proposition}

\begin{proof} First observe that \eqref{eqn:zeros} and our assumptions on the coefficients $a_j, c_j$ imply that $\frac{1}{a_j - c_j } = 1$ and $\tau_{2,j} = -1$  for $j=1,2$.  Corollary ~\ref{thm:nr} implies that
\[ \text{Clos}\left( \mathcal{W}( S^1_{\Theta}) \right) = \text{Conv} \Big( \bigcup_{z \in \mathbb{T}} \mathcal{W}(M_\Theta(z)) \Big) = \text{Conv} \Big( \bigcup_{z \in \mathbb{T}} \mathcal{W}(M_\Theta(\bar{z})) \Big),\]
and to simplify notation, we will often work with $M_\Theta(\bar{z})$. Observe that the circles of foci $\left\{\frac{1}{a_j + c_j z}: \bar{z} \in \mathbb{T} \right\}$ lie in $\mathcal{W}( S^1_{\Theta})$ and cannot contain $\infty$ since $ S^1_{\Theta}$ is a contraction. The circles pass through the points $\frac{1}{a_j + c_j} \in \mathbb{R},$ when $z = 1$, and $\frac{1}{a_j - c_j} = 1,$ when $z = -1$. 

Now we show that $0  \in \mathcal{W}( S^1_{\Theta})^0$ if and only if $c_1 c_2 > \frac{1}{2}.$ As pointed out after \eqref{eqn:pos12},  $M_\Theta(\bar{z})$ has the same numerical range as 
\[
\widehat{M}_{\Theta}(\bar{z}) =
  \begin{bmatrix}
    \frac{1}{a_2 + c_2 z}& \sqrt{|a_1 a_2 c_1 c_2|} \frac{|z +1|^2}{|a_1 + c_1 z|\,|a_2 + c_2 z|}\\
   0 & \frac{1}{a_1 + c_1 z}
  \end{bmatrix}
\]
and so, we work with $\widehat{M}_{\Theta}(\bar{z})$. In particular $H(M_\Theta(\bar{z})):= \frac{1}{2} \left(\widehat{M}_{\Theta}(\bar{z}) + \widehat{M}_{\Theta}(\bar{z})^*\right)$ is a Hermitian matrix and therefore its numerical range is a real line segment. The endpoints are the minimum and maximum eigenvalues of $H(M_\Theta(\bar{z}))$, see  \cite[p.12]{HJ} or \cite{K08}. Furthermore, $\mathcal{W}(H(M_\Theta(\bar{z})))$ is the projection of $\mathcal{W}\big(\widehat{M}_{\Theta}(\bar{z})\big)$ and hence, of $\mathcal{W}(M_\Theta(\bar{z}))$, onto the real axis.  We now study the eigenvalues of $H(M_\Theta(\bar{z}))$, which give the minimum and maximum real parts of the elements in $\mathcal{W}(M_\Theta(\bar{z}))$.



First, the trace of $H(M_\Theta(\bar{z}))$, which is the sum of the two eigenvalues of $H(M_\Theta(\bar{z}))$, equals
\[\frac{a_1 + c_1 \Re z}{|a_1 + c_1 z|^2}+\frac{a_2 + c_2 \Re z}{|a_2 + c_2 z|^2} >0,\]  since $a_j - c_j = 1$. This shows that at least one eigenvalue of $H(M_\Theta(\bar{z}))$ is positive. Then, the minimum eigenvalue will be negative if and only if $\det(H(M_\Theta(\bar{z})) < 0$. In this case, we have
\[\det \left(H(M_\Theta(\bar{z}))\right) =
\det\begin{bmatrix}
    \frac{a_2 + c_2 \Re z}{|a_2 + c_2 z|^2}& \frac{\sqrt{|a_1 a_2 c_1 c_2|}}{2} \left(\frac{|z +1|^2}{|a_1 + c_1 z|\,|a_2 + c_2 z|}\right)\\
    &\\
\frac{\sqrt{|a_1 a_2 c_1 c_2|}}{2} \left(\frac{|z +1|^2}{|a_1 + c_1 z|\,|a_2 + c_2 z|}\right) & \frac{a_1 + c_1 \Re z}{|a_1 + c_1 z|^2}
  \end{bmatrix}.
\]
Let $x = \Re z$. Then some $H(M_\Theta(\bar{z}))$ will have a negative eigenvalue if and only if there exists  $x \in (-1, 1]$ with
\[f(x) = (1 + c_1 + c_1 x)(1 + c_2 + c_2 x) - ((c_1+c_1^2)(c_2 + c_2^2))(1+ x)^2 < 0.\]
The two zeros of $f$ occur at
$$\frac{1}{c_1 c_2} - 1 ~\mbox{and}~ -1 - \frac{1}{c_1 + c_2 + c_1 c_2}.$$
Thus, $f$ has a zero between $-1$ and $1$ if and only if $c_1 c_2 > \frac{1}{2}$ and $f$ will be negative at some point $x \in (-1, 1)$ if and only if one zero lies strictly between $-1$ and $1$. Therefore:

\begin{enumerate}
\item \label{item:1} If $c_1 c_2 < \frac{1}{2}$, then there is no such value of $x$. This implies that for each $z \in \mathbb{T}$,  the matrix $H(M_\Theta(\bar{z}))$ has only positive eigenvalues and so, $\mathcal{W}(M_\Theta(\bar{z})) \subseteq \{x + i y: x > 0\}$. From this, we can conclude that $0 \notin \mathcal{W}( S^1_{\Theta})^0$.

\smallskip
\item If $c_1 c_2 > \frac{1}{2}$, then $f$ is negative at some point strictly between $\frac{1}{c_1 c_2} - 1$ and $1$. Therefore, for some $z_0 \in \mathbb{T}$ (with $z_0 \ne \pm 1$) one eigenvalue of $H(M_\Theta(\bar{z}_0))$ is positive and one is negative. Thus, $\mathcal{W}(M_\Theta(\bar{z}_0))$ contains a point $\lambda_{z_0}$ with negative real part.

Recall that the numerical range of any $M_\Theta(z)$ is the elliptical disk with foci at $ \frac{1}{a_j + c_j \bar{z}}$ and minor axis of length 
\[ \sqrt{|a_1 a_2 c_1 c_2|} \frac{|z +1|^2}{|a_1 + c_1 z|\,|a_2 + c_2 z|}.\]
This implies that $\mathcal{W}(M_\Theta(z_0))$ is the reflection of $\mathcal{W}(M_\Theta(\bar{z}_0))$ across the $x$-axis and thus, $\bar{\lambda}_{z_0} \in \mathcal{W}( S^1_{\Theta})$. If $\lambda_{z_0} \notin \mathbb{R}$, the triangle joining $\lambda_{z_0}, \bar{\lambda}_{z_0}$ and $1$ must be contained in $\mathcal{W}( S^1_{\Theta})$, which implies $0 \in \mathcal{W}( S^1_{\Theta})^0$.

Now let $\lambda_{z_0} \in \mathbb{R}$. By assumption, we also have $\lambda_{z_0}$ negative.
By earlier arguments, the circle $\big\{\frac{1}{a_1 + c_1 z}: z \in \mathbb{T}\big\} \subset \mathcal{W}( S^1_{\Theta})$. This circle passes through the points $1$ and $\frac{1}{a_1 + c_1}$ so it contains points in the first and fourth quadrants. Denote two such points by $\lambda_I$ and $\lambda_{IV}$. Then, the triangle joining $\lambda_{z_0}$, $\lambda_I$, and $\lambda_{IV}$ is contained in the numerical range and so $0 \in \mathcal{W}( S^1_{\Theta})^0$.

\smallskip

 
\item If $c_1 c_2 = \frac{1}{2}$, then $f(x) > 0$ for all $x \in (-1, 1)$ and there are no values in any $\mathcal{W}(M_\Theta(\bar{z}))$ with negative real part; i.e., $\bigcup_{z \in \mathbb{T}} \mathcal{W}(M_\Theta(\bar{z})) \subseteq \{z \in \mathbb{C}: z = x + i y, x \ge 0\}$. On the other hand, if we consider $z = 1$, we can see that zero satisfies the equation 
\[\left|0 - \frac{1}{a_1 + c_1}\right| + \left|0 - \frac{1}{a_2 + c_2}\right| = \left|1 - \left(\frac{1}{a_1 +  c_1}\right)\left(\frac{1}{a_2 + c_2}\right)\right|.\] Thus $0 \in \mathcal{W}(M_\Theta(1))$ and therefore $0 \in \partial \mathcal{W}( S^1_{\Theta})$.
 \end{enumerate}

From these arguments, we know that if $c_1 c_2 > \frac{1}{2}$, then $0 \in \mathcal{W}( S^1_{\Theta})^0$ and if $c_1 c_2 < \frac{1}{2}$, then $0 \notin \mathcal{W}( S^1_{\Theta})^0$. Furthermore, if $c_1 c_2 = \frac{1}{2}$, then $0 \in \partial \mathcal{W}( S^1_{\Theta})$. Thus, we have proven most of Proposition~\ref{lem:positive_eigenvalues}. It just remains to show that if $0 \in \partial \mathcal{W}( S^1_{\Theta})$, then $c_1 c_2 = \frac{1}{2}$.


Assume  $0 \in \partial \mathcal{W}( S^1_{\Theta})$. If $c_1 c_2 > \frac{1}{2}$, then $0 \in \mathcal{W}( S^1_{\Theta})^0$, a contradiction. If 
$c_1 c_2 < \frac{1}{2}$, the zeros of $f$ are at 
\[\frac{1}{c_1 c_2} - 1 >1 ~\mbox{ and }~ -1 - \frac{1}{c_1 + c_2 + c_1 c_2} <-1.\]
 Since $f(x) = \alpha_1 + \alpha_2 x + \alpha_3 x^2$ with $\alpha_3 = c_1 c_2 - c_1 c_2 (1+c_1)(1+c_2) < 0$, the minimum value of $f$ on $[-1,1]$  must be either $f(-1) = 1$ or \[f(1) = (1 + 2c_1)(1 + 2c_2) - 4 c_1 c_2(1+c_1)(1+c_2) = (1 - 4 c_1^2 c_2^2)+2c_1(1 - 2c_1 c_2)+2c_2(1-2c_1c_2) >  0.\] 
Now define the quantity
 \[m:=\min \left \{f(-1), f(1) \right \} > 0.\]
Fix $z \in \mathbb{T}$ and let $\lambda_1, \lambda_2$ be the two eigenvalues of $H(M_\Theta(\bar{z}))$. Since $ S^1_{\Theta}$ is a contraction, we know $\lambda_1, \lambda_2 \le 1.$ Without loss of generality, assume $\lambda_1 = \min \{\lambda_1, \lambda_2\}$. By assumption, $\lambda_2 \ne 0$, since $f(\Re z) \ne 0$. Then we can conclude that 
 \[ \lambda_1 = \frac{ \det \left( H(M_\Theta(\bar{z})\right) }{\lambda_2} \ge  \det \left( H(M_\Theta(\bar{z})\right) \ge m.\]
This immediately implies that for each $z \in \mathbb{T}$, we have $\mathcal{W}\left(M_{\Theta}(\bar{z})\right) \subseteq \{x + i y: x \ge m \}$ and zero cannot lie in the convex hull of the union of these sets. So, if zero lies in the boundary of the numerical range, then $c_1 c_2 = \frac{1}{2}$.
\end{proof}

\section{Boundary of the Numerical Range} \label{sec:boundary}

\subsection{Initial Reductions and Formulas}
We now analyze the boundary of $\mathcal{W}( S^1_{\Theta})$ or equivalently, the boundary of $\text{Clos} (\mathcal{W}( S^1_{\Theta})),$ for a special class of rational inner functions. Specifically, let $\Theta = \theta^2_1$, where $\theta_1$ has a zero on $\mathbb{T}^2$ and $\theta_1 = \frac{\tilde{p}}{p}$ for $p(z)=a -z_1 + c z_2$ with $a,c \ne 0$. The following remark shows that, without loss of further generality, we can assume $a, c >0.$

\begin{remark} Assume $\theta_1 = \frac{\tilde{p}}{p}$ for some $p(z)=a -z_1 + c z_2$ and set $\Theta = \theta_1^2.$ Then by Corollary~\ref{thm:nr},
\[\text{Clos}\left( \mathcal{W} \left( S^1_{\Theta}\right) \right) = \text{Conv}\Big( \bigcup_{\tau \in \mathbb{T}} \mathcal{W}(M_{\Theta}(\tau))\Big),\] where $M_{\Theta}$ is the $2\times 2$ matrix-valued function from Theorem \ref{thm:nr1}. Now write  $a = |a| e^{i \theta}$, $c = |c| e^{i \psi}$, and $w = e^{i (\psi - \theta)}z$ and observe that \eqref{eqn:zeros} implies that $\tau_2 = - e^{i (\theta - \psi)}$. With these substitutions, $M_{\Theta}^*$ changes from
\[
M^*_\Theta(z) =
  \begin{bmatrix}
    \frac{1}{a + c z}& \bar{a}c \frac{(z -\tau_2)^2}{(a + c z)^2}\\
   0 & \frac{1}{a + c z}
  \end{bmatrix}
\ \text{ to } \ 
\widetilde{M}_{\Theta}^*(w)=
 e^{-i\theta} \begin{bmatrix}
    \frac{1}{|a| + |c|w}& |ac|e^{-i\psi} \frac{(w +1)^2}{(|a| + |c| w)^2}\\
   0 & \frac{1}{|a| + |c| w}
  \end{bmatrix}.
\]
Since, when computing numerical ranges, the variables $z$ and $w$ above will take on all values in $\mathbb{T}$, we can conclude 
\[\text{Clos}\left( \mathcal{W} \left( S^1_{\Theta}\right) \right) = \text{Conv}\Big( \bigcup_{\tau \in \mathbb{T}} \mathcal{W}(M_{\Theta}(\tau))\Big) =  \text{Conv}\Big( \bigcup_{\tau \in \mathbb{T}} \mathcal{W}(\widetilde{M}_{\Theta}(\tau))\Big).\]
Thus, if we set $q(z) = |a| -z_1 + |c| z_2$ and $\phi_1 = \frac{\tilde{q}}{q}$ and define $\Phi = \phi_1^2$, then $\text{Clos} (\mathcal{W}( S^1_{\Phi}))$  equals $\text{Clos} (\mathcal{W}( S^1_{\Theta}))$.
\end{remark}

Henceforth, we assume that $p(z)= a -z_1 +cz_2$ where $a, c >0$. By \eqref{eqn:zeros} and Proposition~\ref{prop:elementary}, this forces $\tau_2=-1$ and $a-c=1$.  Furthermore, by the Elliptical Range Theorem, the boundary of each $\mathcal{W}(M_{\Theta}(\tau))$ is a circle with center $c_{\tau}:=\frac{1}{a + c \bar{\tau}}$ and radius $r_{\tau}$ half the modulus of the $(2,1)$-entry of $M_{\Theta}(\tau)$.  Thus, we need to understand a family of circles. For later computations, we require the following alternate parameterization.

\begin{remark} \label{rem:circles} The set of circles $\left \{ \partial \mathcal{W}\big( M_{\Theta}(\tau)\big) \right\}_{\tau \in \mathbb{T}}$  is equal to the set of circles $\{ \mathcal{C}_{\theta}\}_{\theta \in [0,2\pi)}$, where each $\mathcal{C}_{\theta}$ has  center and radius given by
\begin{equation}
\label{eqn:circle} c(\theta) :=  \frac{ a + ce^{i \theta}}{a+c} \ \ \text{ and } \ \ 
r(\theta) :=  \frac{ac}{(a+c)^2}\left(1- \cos \theta\right).\end{equation}
To see this, define the Blaschke factor
\[ B(z): = - \frac{ \frac{c}{a} +z}{1 + \frac{c}{a} z } = - \frac{ c +a z }{a +c z}.\]
Then $B$ maps $\mathbb{T}$ one-to-one and onto itself. Now fix $\tau \in \mathbb{T}$, set $\lambda = B(\bar{\tau}) \in \mathbb{T}$, and choose $\theta$ to be the unique
angle in $[0, 2\pi)$ with $\lambda = e^{i\theta}.$ Observe that 
\[ a + c \lambda = \frac{ a ( a +c \bar{\tau}) -c(c +a \bar{\tau})}{a + c\bar{\tau}} = \frac{(a+c)(a-c)}{a + c \bar{\tau}} = \frac{ a+c}{a+c\bar{\tau}},  \] 
where we used $a-c=1$. Then the center of $\partial \mathcal{W}\big( M_{\Theta}(\tau))$ is 
\[ c_{\tau} = \frac{1}{a+c \bar{\tau}} \cdot \frac{ a+c\lambda}{a+c\lambda} =  \frac{a+c\lambda}{a+c} = c(\theta).\]
To consider the radius, first observe that since $\lambda \in \mathbb{T}$, we have $2(1-\cos \theta ) = \left | 1 - \lambda \right|^2.$ Moreover
\[ \left | 1 - \lambda \right|^2 = \left | 1  + \frac{ c +a \bar{\tau} }{a +c \bar{\tau}} \right|^2 = (a+c)^2 \left |\frac{ 1 + \bar{\tau}}{a + c\bar{\tau}} \right |^2. \]
Using that equation, we can write the radius of $\partial \mathcal{W}\big( M_{\Theta}(\tau))$ as
\[ r_{\tau} = \frac{ac}{2}  \left | \frac{ 1 + \bar{\tau}}{a + c \bar{\tau}}\right|^2 = \frac{ac\left | 1 - \lambda \right|^2}{2(a+c)^2}  = \frac{ac(1 -\cos \theta)}{(a+c)^2} = r(\theta), \]
which proves the  claim.
\end {remark}

\subsection{Circular Numerical Ranges}

In the one-variable situation, if $B$ is a degree-$2$ Blaschke product, then the numerical range of $S_B$ is circular disk if and only if the two zeros of $B$ are the same. One might conjecture that a similar statement should hold in two variables, namely if $\theta = \theta_1^2$, then Clos$(\mathcal{W}( S^1_{\Theta}))$ is circular. In this section, we show this is not the case.

Now fix $\tau \in \mathbb{T}$. Then $\partial \mathcal{W}(M_{\Theta}(\tau))$ is a circle with radius
\[r_{\tau} = \frac{ac}{2} \frac{|\bar{\tau} + 1|^2}{|a + c\bar{\tau}|^2} = ac \frac{1 + x}{(a + c x)^2 + c^2(1-x^2)},\] 
where $\tau = x + i y$. One can check that $r_{\tau}$ increases as $x$ increases. Therefore, the maximum and minimum values of $r_{\tau}$
occur when $\tau = 1$ and $\tau = -1$, respectively.  Now consider the alternate formulas given in \eqref{eqn:circle}. First, since the centers are exactly the points
\[ c(\theta) = \frac{a+ce^{i\theta} }{a+c} = \left( \frac{a}{a+c} + \frac{c}{a+c}\cos \theta, \ \frac{c}{a+c}\sin \theta\right),\]
$c(\theta)$ and $c(2\pi-\theta)$ are reflections of each other across the real axis. Moreover, \eqref{eqn:circle} also implies that $r(\theta) = r(2\pi-\theta)$ and so
 the set of circles $\{\mathcal{C}_{\theta}\}_{\theta \in [0,2\pi)}$ is symmetric with respect to the real axis. This immediately implies  Clos$(\mathcal{W}( S^1_{\Theta}))$ must also be symmetric with respect to the real axis.

Thus, if Clos$(\mathcal{W}( S^1_{\Theta}))$ were circular, the real line would contain the diameter. Furthermore, the value $\frac{1}{a-c} = 1$ obtained when $\tau = -1$ is in the numerical range and the numerical radius is $1$. So $1$ is the maximum value on the real axis. The smallest value on the real axis occurs when $\tau= 1$ or equivalently, when $\theta=\pi.$  Then \eqref{eqn:circle} shows that 
the center $c_{1}= \frac{1}{a+c}$ is real and has real part smaller than any other $c_{\tau}$. Similarly, the radius $r_{1} =\frac{2ac}{(a+c)^2}$ is maximal and so, the smallest value of Clos$(\mathcal{W}( S^1_{\Theta}))$ on the real axis is
 \[
 \frac{1}{a + c} -\frac{2 ac}{(a + c)^2}=  \frac{a + c - 2 a c}{(a+c)^2}.
 \] 
Thus, these are the extreme real values of the numerical range, and if the numerical range were circular, they would be the endpoints of a diameter. Then the center of the circle would be the point $(\alpha, 0)$, with $\alpha$  given by
\begin{align}
 \alpha &:= \frac{1}{2}\left(1 + \frac{a + c - 2 a c}{(a+c)^2}\right) 
= \frac{a^2 + a + c(1+c)}{2(a+c)^2}
= \frac{a^2 + a + (a-1)a}{2(a + c)^2} =  \frac{a^2}{(a+c)^2},
\end{align}
where we used $a=c+1$. Similarly, the radius $r$ would be
\begin{align*}
r &: =\frac{1}{2}\left(1 - \frac{a + c - 2 a c}{(a+c)^2}\right) = \frac{a^2 - a + 4ac + c^2 - c}{2(a+c)^2} \\
&=\frac{4 a c + c^2 - c+(c+1)^2-(c+1)}{2(a + c)^2}= \frac{2 a c + c^2}{(a + c)^2}. 
\end{align*}
We can now find a point $Q$ that is in the numerical range but is not in that circle. Specifically, consider $\theta = \frac{\pi}{2}$
and using \eqref{eqn:circle}, define the point $Q$ by
\[Q := c\Big(\tfrac{\pi}{2}\Big) + r\Big(\tfrac{\pi}{2}\Big) e^{i\frac{\pi}{2}} = \left(\frac{a}{a + c}, \frac{c^2 + 2ac}{(a+c)^2}\right),\]
which is on $\mathcal{C}_{\frac{\pi}{2}}$ and hence is in Clos$(\mathcal{W}( S^1_{\Theta}))$. If the numerical range were circular, the point $Q$ would lie in or on the circle bounding the numerical range with center $(\alpha,0)$ and radius $r$. Computing the distance from $Q$ to the center gives
\[
~\mbox{dist}^2(Q,(\alpha, 0)) = \left(\frac{a}{a+c} - \frac{a^2}{(a+c)^2}\right)^2+\frac{(c^2 + 2 ac)^2}{(a + c)^4} = \frac{(ac)^2}{(a+c)^4}+\frac{(c^2+2ac)^2}{(a+c)^4}.\] 
For $Q$ to be in the circle, we must have 
\[~\mbox{dist}^2(Q,(\alpha, 0)) \le r^2 = \frac{(2 a c+c^2)^2}{(a + c)^4},\] 
which is impossible.
Thus, Clos$(\mathcal{W}( S^1_{\Theta}))$ cannot be circular.

\subsection{The Boundary of the Numerical Range}

The goal of this section is to prove the following theorem:

\begin{theorem} \label{thm:boundary} Let $\Theta = \theta^2_1$ be a degree $(2,2)$ rational inner function, where $\theta_1 = \frac{\tilde{p}}{p}$ for a polynomial $p(z)=a -z_1 + c z_2$ with no zeros on $\mathbb{D}^2$, a zero on $\mathbb{T}^2$, and $a, c >0$. Then the boundary of $\mathcal{W}( S^1_{\Theta})$ is given by the curve $E = (x(\theta), y(\theta))$ where 
\[ \begin{aligned}
x(\theta) &= \frac{a + c \cos \theta}{a+c}  + \frac{ac(1-\cos \theta)}{(a+c)^2}  \cos\left ( \theta - \text{\emph{arcsin}}\left( \frac{a}{a+c} \sin \theta \right) \right) \\
  y(\theta) &= \frac{c \sin \theta}{a+c}  + \frac{ac(1-\cos \theta)}{(a+c)^2}  \sin\left ( \theta - \text{\emph{arcsin}}\left( \frac{a}{a+c} \sin \theta \right) \right),
\end{aligned}
\]
for $\theta \in [0, 2\pi)$. 
\end{theorem}    

We prove Theorem \ref{thm:boundary} using the theory of envelopes of families of curves. The proof takes a bit of work, so we break it into sections.

\subsubsection{Introduction to Envelopes} Let $f(x,y,\theta)=0$ be a family of (distinct) curves parameterized by $\theta$.
One may think of the {\it envelope} $E$ of a family of curves as a curve that is tangent to each member of the family. There are several competing definitions for the notion of an envelope, one of which is the curve that satisfies the {\it envelope algorithm} that we describe below. We take that as our definition, noting that in this case, the standard ways of thinking about envelopes agree. A discussion of these notions can be found in Courant \cite[p. 171]{Courant}. We also refer readers interested in envelopes to \cite{kalman07}. 

Assume the family of curves $f(x,y,\theta)=0$ satisfies $f_x^2 + f_y^2 \ne 0.$ Let $E$ be a curve parameterized as $(x(\theta), y(\theta))$  where $x(\theta)$ and $y(\theta)$ are continuously differentiable functions. Then we say that $E$ satisfies the {\it envelope algorithm} if the points on $E$ satisfy the equations
\begin{equation} \label{eqn:econd} f(x, y, \theta) = 0 \text{ and } f_{\theta}(x, y, \theta) = 0 \end{equation} 
and the functions $x(\theta)$ and $y(\theta)$ satisfy 
\begin{equation}\label{derivative_theta} \left(\tfrac{dx}{d\theta}\right)^2 + \left(\tfrac{dy}{d\theta}\right)^2 \ne 0.
\end{equation}
An alternate way to compute an envelope $E$ involves using intersections of the curves $f(x,y,\theta)=0$ associated to different $\theta$. For this method, assume an envelope $E$ exists and can be  parameterized as $(x(\theta), y(\theta))$  for $x(\theta), y(\theta)$ continuously differentiable functions satisfying \eqref{derivative_theta}. Then,  fix $h$ and $\theta$ and locate the intersection point of the curves $f(x, y, \theta+h) = 0$ and $f(x, y, \theta) = 0$; call this point $p_{h,\theta}.$ Then $p_{\theta}:= \lim_{h \rightarrow 0} p_{h, \theta}$ gives the point on the envelope $E$ tangent to the curve $f(x,y,\theta)=0$.

\subsubsection{Notation and Summary.} 
To study $\mathcal{W}( S^1_{\Theta})$, Corollary \ref{thm:nr} implies that we need to study the family of circles $\{ \partial \mathcal{W}(M_{\Theta}(\tau))\}_{\tau \in \mathbb{T}}$. By Remark \ref{rem:circles}, it is equivalent to consider the family of circles $\{ \mathcal{C}_{\theta} \}_{\theta \in [0, 2\pi)}$, where
each $\mathcal{C}_{\theta}$ has center and radius given by
\[ 
\begin{aligned}
c(\theta) &=  c_1(\theta) + i c_2(\theta) = \frac{ a+c \cos \theta}{a+c} +i \frac{c \sin \theta}{a+c}; \\
 r(\theta) &=  \frac{ac}{(a+c)^2}\left(1- \cos \theta\right).
 \end{aligned}
 \]
 To align with the envelope notation, observe that the family of circles $\{C_{\theta}\}_{\theta \in [0,2\pi)}$ is also the set of curves satisfying $f(x,y,\theta)=0$ for
\begin{equation} \label{eqn:Ctheta} f(x,y,\theta) = \left( x- c_1(\theta) \right)^2 +  \left( y- c_2(\theta) \right)^2 - r(\theta)^2, \quad \theta \in [0, 2\pi). \end{equation}
For each $\theta \in [0, 2\pi)$, let $\mathcal{D}_{\theta}$ denote
 the open disk with boundary $\mathcal{C}_{\theta}$. Let $\mathcal{C}: = \{c(\theta) : \theta \in [0, 2\pi)\} $ denote the circle of centers of the 
 $\mathcal{C}_{\theta}$ and let $\mathcal{D}$ denote the open disk with boundary $\mathcal{C}.$   
 Set $\Omega = \mathcal{D} \cup \bigcup_{\theta \in [0, 2\pi)} \mathcal{D}_{\theta}$ and let $\mathcal{B}$ denote the boundary of $\Omega.$ Then the closure of the numerical 
 range $\mathcal{W}( S^1_{\Theta})$ is the closed convex hull of $\Omega.$  

In what follows, we find an envelope of the family of curves $\{\mathcal{C}_{\theta}\}_{\theta \in [0,2\pi)}$ and use it to compute the boundary of 
$\mathcal{W}( S^1_{\Theta})$. First, observe that our family of curves satisfies $f_x^2 + f_y^2 \ne 0$ for $\theta \ne 0.$ Then to find an envelope of 
$\{\mathcal{C}_{\theta}\}_{\theta \in [0,2\pi)}$, we need only find a curve $E$ satisfying  \eqref{eqn:econd} and \eqref{derivative_theta}.
Specifically, we will find all points satisfying \eqref{eqn:econd}. These points 
 will yield two curves $E_1$ and $E_2$. We will show $E_1$ also satisfies \eqref{derivative_theta} and thus, gives an envelope for
 our family of curves. We further show that $E_1$ is a convex curve bounding the set $\Omega.$ This implies $\Omega$ is convex
 and so $\overline{\Omega} =$ Clos$(\mathcal{W}( S^1_{\Theta}))$. Thus $E_1$ gives the boundary of Clos$(\mathcal{W}( S^1_{\Theta})),$ and hence of $\mathcal{W}( S^1_{\Theta})$,
 as desired.
 
\subsubsection{Finding the Envelope.} We first identify all points satisfying \eqref{eqn:econd}, which
gives the two equations $f(x,y,\theta)=0$
and 
\begin{equation} \label{eqn:deriv_theta2} -\left( x- c_1(\theta) \right) c_1'(\theta)  -  \left( y- c_2(\theta) \right) c_2'(\theta)  - r(\theta) r'(\theta)=0.\end{equation}
Observe that we can write each circle $\mathcal{C}_{\theta}$ parametrically as
\[ x(s) = c_1(\theta) + r(\theta) \cos(s), \ \  y(s) =c_2(\theta) + r(\theta) \sin (s), \ \ s\in [0, 2\pi).\] 
Then \eqref{eqn:deriv_theta2} is equivalent to 
\[- r(\theta) \cos(s) \frac{c \sin \theta}{a+c} + r(\theta) \sin(s) \frac{c \cos \theta}{a+c} + r(\theta) \frac{ac}{(a+c)^2} \sin \theta=0.\]
For $\theta \ne 0$, we have $r(\theta) \ne 0$ and so, this is equivalent to 
\begin{equation} \label{eqn:s} \sin(s-\theta) = \cos \theta \sin(s) - \sin \theta \cos(s) = -\frac{a}{a+c} \sin \theta.\end{equation}
Note that the above equation has two solutions for $s$:
\begin{equation} \label{eqn:sj}  s_1(\theta) := \theta - \text{arcsin}\left( \frac{a}{a+c} \sin \theta\right)  \ \text{ and } \ s_2(\theta) :=  \theta - \pi + \text{arcsin}\left( \frac{a}{a+c} \sin \theta\right).\end{equation}
Then the curves $E_1(\theta)=(x_1(\theta), y_1(\theta)) $ and $E_2(\theta)=(x_2(\theta), y_2(\theta)) $ defined by  
\[ 
x_j(\theta) := c_1(\theta) + r(\theta)  \cos( s_j(\theta)) \text{ and }  y_j(\theta) := c_2(\theta) + r(\theta)  \sin( s_j(\theta)),  \ \ \theta \in (0, 2\pi), j=1,2 \\
\]
give two curves whose points satisfy \eqref{eqn:econd}. 

Since we are concerned with the convex hull of the family of circles $\{\mathcal{C}_{\theta}\}_{\theta\in [0, 2\pi)}$, we consider the outer curve $E_1$. To show that $E_1$ satisfies \eqref{derivative_theta}, we need to do a little more work. 
First, observe that \eqref{eqn:s} implies the following two equations:
\begin{equation} \label{eqn:useful} \cos\big( \theta - s_1(\theta) \big)  \left( 1-s_1'(\theta)\right) = \frac{a}{a+c} \cos \theta \ \ \text{ and } \ \  \Re \left(c^\prime(\theta) e^{-i s_1(\theta)}\right) = - r^\prime(\theta).\end{equation}
We can obtain more information by writing 
\[ E_1(\theta) = c(\theta) + r(\theta) e^{i s_1(\theta)} = \frac{a+ce^{i\theta}}{a+c} + \frac{ac(1-\cos \theta)}{(a+c)^2} e^{i s_1(\theta)}\]
and then computing derivatives as follows:
\[
\begin{aligned}
x_1^\prime(\theta) + i y_1^\prime(\theta) &= e^{is_1(\theta)} \left(c^\prime(\theta) e^{-i s_1(\theta)} + r^\prime(\theta) + i r(\theta) s_1'(\theta)\right)\\
&= e^{is_1(\theta)} \left(\Re\left(c^\prime(\theta) e^{-is_1(\theta)} + r^\prime(\theta)\right) + i \Im \left(c^\prime(\theta) e^{-is_1(\theta)} + r^\prime(\theta)\right) + i r(\theta) s_1'(\theta) \right).
\end{aligned}
\]
Then, using \eqref{eqn:useful} and the fact that $r^\prime(\theta)$ is real, we have
\[
\begin{aligned}
x_1^\prime(\theta) + i y_1^\prime(\theta) &= i e^{is_1(\theta)} \left(\Im\left(\frac{c}{a + c} i e^{i \theta} e^{-is_1(\theta)} + r^\prime(\theta)\right) + r(\theta) s_1'(\theta) \right)\\
&=i e^{is_1(\theta)}\left(\frac{c}{a+c}\cos(\theta - s_1(\theta) ) + \frac{ac}{(a + c)^2}(1 - \cos(\theta))s_1'(\theta) \right),
 \end{aligned}
 \]
which allows us to conclude that
\begin{eqnarray}
 \label{eqn:x_derivative} x_1^\prime(\theta) &=& -\sin \left( s_1(\theta) \right) \left(\frac{c}{a+c} \cos\left(\theta - s_1(\theta) \right) + \frac{ac}{(a+c)^2} (1 - \cos\theta)s_1'(\theta) \right); \\
 \label{eqn:y_derivative} y_1^\prime(\theta) &=& \cos \left( s_1(\theta) \right) \left(\frac{c}{a+c} \cos\left (\theta - s_1(\theta) \right) + \frac{ac}{(a+c)^2} (1 - \cos\theta)s_1'(\theta)\right).
 \end{eqnarray}
To conclude  \eqref{derivative_theta} for $E_1$, one just needs to show that 
\begin{equation} \label{eqn:deriv_zero} \frac{c}{a+c} \cos\left (\theta - s_1(\theta) \right) + \frac{ac}{(a+c)^2} (1 - \cos\theta)s_1'(\theta) \ne 0,\end{equation}
for  $\theta \ne 0.$ This is almost immediate. First observe that since $\left| \frac{a}{a+c}\right| <1$, for $\theta \in [0, 2 \pi)$,
\[ -\frac{\pi}{2} < \text{arcsin}\left( \frac{a}{a+c} \sin \theta \right) < \frac{\pi}{2},\]
and so $ \theta - s_1(\theta) \in (- \frac{\pi}{2}, \frac{\pi}{2})$. This implies $\cos( \theta - s_1(\theta)) >0$. Moreover, one can compute
\[ s_1'(\theta) = 1 - \frac{a}{a+c} \frac{ \cos \theta}{ \sqrt{ 1 - \frac{a^2}{(a+c)^2} \sin^2 \theta}} \] 
and observe that $s_1'$ is continuous and $s'_1(\frac{\pi}{2}) = 1 >0$. One can show that $s_1'(\theta)=0$ leads to the contradiction $ 1 = \frac{a}{a+c}$. 
Thus, $s_1'>0$ as well and we can conclude that \eqref{eqn:deriv_zero} is strictly positive. This implies $E_1$ satisfies  \eqref{derivative_theta} and thus,
is an envelope for the family $\{ \mathcal{C}_{\theta} \}_{\theta \in [0, 2\pi)}$. 

Finally, a word about $\theta=0$. Because the circle $\mathcal{C}_{0}$ is the single point $(1,0)$, it does not make sense to say a curve is tangent to
$\mathcal{C}_{0}$. However, the formulas $(x_j(\theta), y_j(\theta))$ for each $E_j$ extend to continuously differentiable functions on intervals containing zero in their interior.
In particular,  we can certainly extend $E_1$ and $E_2$ to $\theta=0$ by specifying $E_j(0)= 1$ for $j=1,2$.

\subsubsection{Location of $E_1, E_2$} Let us briefly consider the relationship between the curves $E_1, E_2$ 
and intersections of the circles $\{\mathcal{C}_{\theta}\}_{\theta \in [0,2\pi)}$. We use this relationship to show that with the
exception of the point $(1,0)$, the curve $E_1$ lies completely outside of $\overline{\mathcal{D}}$ and the curve $E_2$ lies completely in the interior of $\mathcal{D}.$

Fix $\theta \ne0$. Then for $h$ with $|h|$ sufficiently small, the  circles $\mathcal{C}_{\theta}$ and $\mathcal{C}_{\theta +h}$ intersect in two points. To verify this, observe that the disks $\mathcal{D}_{\theta}$ and $\mathcal{D}_{\theta +h}$ will overlap for $|h|$ sufficiently small. Moreover, the  circle formula \eqref{eqn:Ctheta} paired with the formulas for $c(\theta)$ and $r(\theta)$ can be used to show that no circle $\mathcal{C}_{\theta}$ is fully contained in a different circle $\mathcal{C}_{\tilde{\theta}}.$ Thus, there must be two intersection points; call them $p^1_{\theta, h}$ and $p^2_{\theta, h}$.

Basic geometry shows that the points $p^1_{\theta, h}$ and $p^2_{\theta, h}$ will be symmetric across the straight line connecting the centers $c(\theta)$ and $c(\theta +h)$. Since $r(\theta) \ne 0$, we can conclude that one point, say $p^1_{\theta, h}$, is in $\overline{\mathcal{D}}^c$ and the other point $p^2_{\theta, h}$ is in $\mathcal{D}.$ Now write the intersection points as 
\[ p_{\theta,h}^j = c(\theta) + r(\theta)e^{i {s^j_h}},\]
where $s^j_h$ is an angle depending on $j$ and $h$. Substituting this formula for  $p_{\theta,h}^j$ into the equation for $\mathcal{C}_{\theta+h}$ gives:
\[  \Big( c_1(\theta) + r(\theta)\cos( s^j_h)- c_1(\theta +h) \Big)^2 +  \Big( c_2(\theta) + r(\theta)\sin( s^j_h)- c_2(\theta+h) \Big)^2 - r(\theta+h)^2= 0, \]
and one can use trigonometric approximations to show that 
\[ \lim_{h \rightarrow 0 }  \sin(s_h^j-\theta) = -\frac{a}{a+c} \sin \theta.\]
This shows that the sets $\{ p_{\theta,h}^1 \}, \{ p_{\theta,h}^2 \}$ converge to points $p^1_{\theta}$ and $p^2_{\theta}$ on $\mathcal{C}_{\theta}$ satisfying \eqref{eqn:s}. This implies that the sets $\{ p^1_{\theta}, p^2_{\theta} \} $ and $\{ E_1(\theta), E_2(\theta)\}$ are equal.

Now we can examine the location of the curves $E_1$ and $E_2.$ First since $E_1(\theta)$ and $E_2(\theta)$ are limits of the $\{p^j_{\theta,h}\}$, they are symmetric points across $\mathcal{C}$. 
Thus, if either of $E_1(\theta)$, $E_2(\theta)$ is on $\mathcal{C}$, we must have $E_1(\theta)=E_2(\theta).$ However, using \eqref{eqn:sj}, one can show that $E_1(\theta)=E_2(\theta)$ only at $\theta=0$. Thus, $E_1$ and $E_2$ only touch $\mathcal{C}$ at $\theta=0.$

Then by the properties of $p^1_{\theta}$ and $p^2_{\theta}$, except at $\theta=0$, one of the curves $E_1, E_2$ is always in  $\overline{\mathcal{D}}^c$ and one is always in $\mathcal{D}.$ By checking at $\theta =\pi$, we can conclude 
\[ 
\begin{aligned}
E_1(\theta) &= p^1_{\theta} \  \in  \overline{\mathcal{D}}^c \ \  \text{ for } 0 < \theta < 2 \pi, \quad E_1(0) = (1,0); \\
E_2(\theta) &= p^2_{\theta} \  \in  {\mathcal{D}} \ \  \text{ for } 0 < \theta < 2 \pi, \quad E_2(0) = (1,0). 
\end{aligned}
\]

\subsubsection{The Boundary of $\Omega$.} Recall that $\mathcal{B}$ denotes the boundary of $\Omega = \mathcal{D} \cup \bigcup_{\theta \in [0, 2\pi)} \mathcal{D}_{\theta}$. We will show that $\mathcal{B} = E_1$. Our initial goal is to show $\mathcal{B} \subseteq E_1$. First, it is easy to conclude that $\mathcal{B} \subseteq \cup_{\theta \in [0, 2\pi)} \mathcal{C}_{\theta}.$ To see this, note that $\mathcal{B}$ is in the boundary of $ \bigcup_{\theta \in [0, 2\pi)} \mathcal{D}_{\theta}$. Then if $\{ c(\theta_n) + \lambda_n r(\theta_n )e^{i s_n}\}$ with $0\le \lambda_n \le 1$ is a sequence converging to a point on $\mathcal{B}$, one can use convergent subsequences of the $\{\theta_n\}$, $\{s_n\}$, and $\{\lambda_n\}$ to conclude that it must converge to a point on some $\mathcal{C}_{\theta}$. 

Since $\Omega$ is in the closure of the numerical range of a contraction, we also know $E_1(0) = (1,0)=\mathcal{C}_0 \in \mathcal{B}.$  Now, we determine the points that the $\mathcal{C}_{\theta}$ with $\theta \ne 0$ can contribute to $\mathcal{B}.$ Fix $\theta \ne 0$. Set $\mathcal{B}_{\theta} = \mathcal{B} \cap \mathcal{C}_{\theta}.$ Further, define
\[ \Omega^{\theta}_N :=  \mathcal{D}_{\theta} \cup \mathcal{D} \cup \left( \bigcup_{\ell=0}^{N-1} \mathcal{D}_{ \frac{2 \pi \ell}{N}} \right).\]
 Let $\mathcal{B}_N$ denote the boundary of $\Omega^{\theta}_N$; then $\mathcal{B}_{N}$ is composed of arcs of circles from the boundaries of the disks comprising $\Omega^{\theta}_N.$ Let $\mathcal{B}_N^{\theta}$ be the contribution of $\mathcal{C}_{\theta}$ to $\mathcal{B}_N$. Since $\Omega$ is open, we know that 
 \[\mathcal{B}^{\theta}_N := \mathcal{B}_{N} \cap \mathcal{C}_{\theta} =  \mathcal{C}_{\theta}  \cap \left(\Omega_N^{\theta} \right)^c.\]
 One can use the definition of boundary and the density of the roots of unity in $\mathbb{T}$ to show 
  \[ \mathcal{B}_{\theta} = \lim_{N \rightarrow \infty} \mathcal{B}^{\theta}_N.\]
Fix $N$ and assume  $\mathcal{B}^{\theta}_N \ne \emptyset.$
By earlier discussions, for $N$ sufficiently large (i.e. the difference between the angles sufficiently small), $\mathcal{C}_{\theta}$ will have one intersection point in $\overline{\mathcal{D}}^c$, call it $p_{\psi}$, with each close $\mathcal{C}_{\psi}$ bounding a disk from $\Omega_N^{\theta}$. Then a whole segment of $\mathcal{C} _{\theta}$ between $p_{\psi}$ and the point on $\mathcal{C}_{\theta} \cap \mathcal{C}$ closest to $c(\psi)$ will be contained in $\mathcal{D}_{\psi}$. This implies that $\mathcal{B}_N^{\theta}$  must be an arc on $\mathcal{C}_{\theta}$ whose endpoints are intersection points of $\mathcal{C}_{\theta}$ and two nearby circles $\mathcal{C}_{\psi_1}$ and $\mathcal{C}_{\psi_2}$.

By earlier remarks about intersection points, as $N \rightarrow \infty$, the intersection points  in $\overline{\mathcal{D}}^c$ between $\mathcal{C}_{\theta}$ and the closest $\mathcal{C}_{\psi}$'s will approach $E_1(\theta).$
Thus we can conclude that either $\mathcal{B}_{\theta} = \emptyset$ or $\mathcal{B}_{\theta} = E_1(\theta).$
This proves the claim that $\mathcal{B}  \subseteq E_1.$ 

To show $\mathcal{B}= E_1$, proceed by contradiction and assume there is some $E_1(\theta)=(x_1(\theta), y_1(\theta)) \not \in \mathcal{B}$. Without loss of generality, assume $0< \theta < \pi$. Earlier arguments showed that $s_1'$ is always positive, so $s_1$ is strictly increasing. Thus, $s_1(\theta) \in (s_1(0), s_1(\pi)) =(0,\pi)$. This implies $\sin(s_1(\theta)) >0$ and by  \eqref{eqn:x_derivative}, $x_1$ is strictly decreasing on $[0, \pi]$. Moreover, on  $(0, \pi)$, we have $y_1>0$ and on $(\pi, 2\pi)$, we have $y_1<0$.  Thus, there is no point on $E_1$ with $x$-coordinate $x_1(\theta)$ and $y$-coordinate strictly larger than $y_1(\theta).$

To obtain the contradiction, define $\alpha = \sup\{ \epsilon : (x_1(\theta), y_1(\theta) + \epsilon) \in \Omega\}$. Since $\Omega$ is bounded, such an $\alpha$ exists and since $E_1(\theta) \not \in \mathcal{B}$, we know $\alpha > 0$. But, then $(x_1(\theta), y_1(\theta) +\alpha) \in \mathcal{B}$ and since $\mathcal{B}\subseteq E_1$, we must have $(x_1(\theta), y_1(\theta) +\alpha) \in E_1$. But, this contradicts our previous statement about $E_1$.
Then it follows that $\mathcal{B} = E_1.$

\subsubsection{The Proof of Theorem \ref{thm:boundary}}
Let $\widehat{\Omega}$ be the closed convex hull of $\Omega$. By previous facts, this implies $\widehat{\Omega} =$Clos$(\mathcal{W}( S^1_{\Theta}))$.
We will show that $E_1$ is the boundary of $\widehat{\Omega}$ and hence, of Clos($\mathcal{W}( S^1_{\Theta}))$ and $\mathcal{W}( S^1_{\Theta})$.

First we show $E_1$ is the boundary of some convex set. To show this, we use the Parallel Tangents condition, which says that a curve $C$ is the boundary of a convex set if and only if there are no three points on $C$ such that the tangents at these points are parallel. Observe that the tangents of $E_1$ are given by $( x_1'(\theta), y_1'(\theta))$ for $\theta \in [0, 2\pi).$ By way of contradiction, assume there are three points whose tangents are parallel, say at $\theta_1, \theta_2, \theta_3 \in [0, 2\pi).$ This implies that  
\begin{equation} \label{eqn:3tangents} \frac{ y_1'(\theta_1)}{x_1'(\theta_1)} =  \frac{ y_1'(\theta_2)}{x_1'(\theta_2)} =\frac{ y_1'(\theta_3)}{x_1'(\theta_3)} .\end{equation}
By \eqref{eqn:x_derivative} and \eqref{eqn:y_derivative}, we know $\frac{ y_1'(\theta)}{x_1'(\theta)} = -\cot(s_1(\theta))$ for $\theta \in [0, 2\pi)$.  Then, since $s_1$ is a one-to-one function mapping $[0, 2\pi)$ onto  $[0, 2\pi)$, Equation \eqref{eqn:3tangents} says that there are three distinct angles $\psi_1, \psi_2, \psi_3 \in [0, 2\pi)$ satisfying
\[ \cot(\psi_1) = \cot(\psi_2) = \cot(\psi_3),\]
which contradicts properties of cotangent. Thus, $E_1$ is the boundary of a convex set $S$. 

As $E_1$ is a bounded closed curve and $S$ is convex, its closure $\overline{S}$ must be the closed convex hull of $E_1$. Similarly, as $\Omega$ is composed of circular disks including $\mathcal{D}$, one can show that $\Omega$ is contained in the closed convex hull of $E_1$. But, then $\Omega \subseteq \overline{S} \subseteq \widehat{\Omega}$, which implies that $\widehat{\Omega} = \overline{S}$. Thus, $E_1$ is the boundary of $\widehat{\Omega}$ and hence, the boundary of Clos$(\mathcal{W}(S^1_{\Theta}))$ and $\mathcal{W}(S_{\Theta}^1)$. \\

Finally, we remark that the boundary of the numerical range is not, in general, the set of extreme points that one obtains from the circles. Here, by an extreme point, we mean the point on $\mathcal{C}_{\theta}$ furthest away from the center of $\mathcal{C}.$  In Figure~\ref{fig:Numerical_range_1_2} for $a=2$ and $c=1$, we present some of the circles $\{\mathcal{C}_{\theta}\}$, the curve consisting of the extreme points of the $\mathcal{C}_{\theta}$, and the boundary of the numerical range of $S_{\Theta}^1.$

\begin{figure}[ht]
\includegraphics[height=6cm]{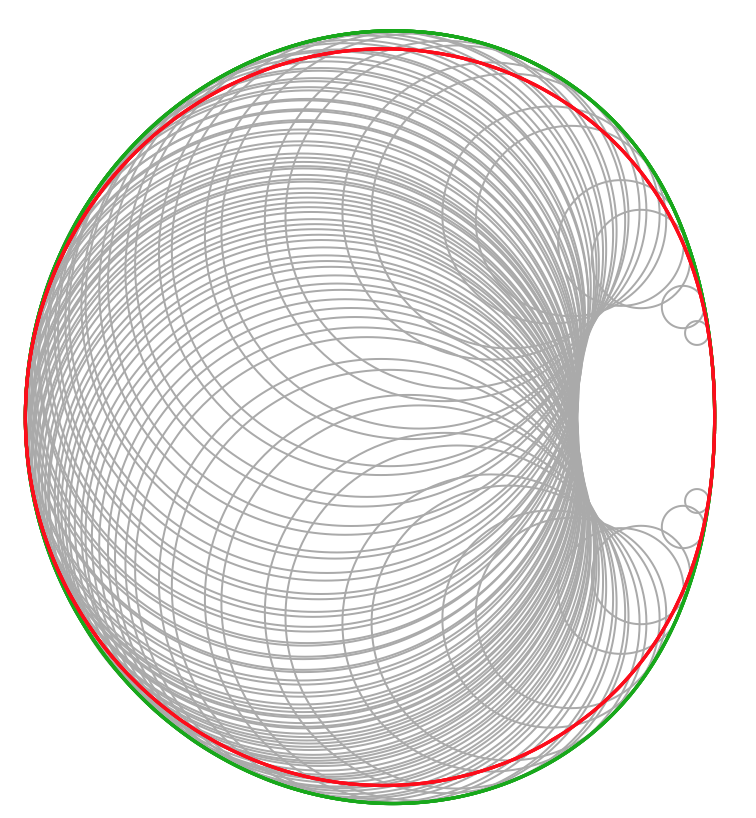}
\caption{The numerical range of $S_{\Theta}^1$ with $a=2$ and $c=1$, the curve of extreme points in red, and the outer envelope of the family of circles in green.}
\label{fig:Numerical_range_1_2}
\end{figure}

\end{document}